\documentclass[11pt, a4paper]{article}

\usepackage{fullpage}
\usepackage{setspace}
\usepackage[section]{placeins}

\usepackage{xcolor}           
\usepackage[normalem]{ulem}    
\usepackage[colorlinks=true, allcolors= blue]{hyperref}
\hypersetup{
    bookmarksdepth=3 %
}

\usepackage{todonotes}
\usepackage[draft, deletedmarkup=sout]{changes} 
\definechangesauthor[name=Rai, color=green]{R}
\definechangesauthor[name=David, color=blue]{D}
\definechangesauthor[name=Krish, color=purple]{K}

\usepackage{amsmath,amssymb,amsthm,mathtools,mathrsfs}
\usepackage{bbold}             
\usepackage{stackengine}       
\usepackage{nicefrac}   
\usepackage{cleveref}

\usepackage{graphicx}
\usepackage{subcaption}
\usepackage{booktabs}
\usepackage{algorithm} 
\usepackage{algpseudocodex} 
\usepackage{algorithmicx} 

\usepackage{tikz}
\usetikzlibrary{
  arrows.meta,
  positioning,
  shapes.geometric,
  decorations.pathreplacing,
  fit,
  shadows.blur,
  automata,
  bayesnet
}
\usetikzlibrary{arrows.meta,automata,positioning,calc}
\usepackage{amsmath,amssymb,amsthm,mathtools,mathrsfs}
\usepackage{bbold}
\usepackage{stackengine}
\usepackage{nicefrac}
\usepackage{cleveref}
\usetikzlibrary{calc,positioning,arrows.meta}
\usetikzlibrary{fit,backgrounds}
\tikzset{
  state/.style={circle,draw,thick,minimum size=15mm,inner sep=0pt,font=\large},
  every edge/.style={draw,thick,-{Latex[length=2.5mm]}}
}
\tikzset{
  p1state/.style={
    state,
    isosceles triangle,
    isosceles triangle apex angle=60,
    shape border rotate=90, % triangle pointing UP
    minimum size=12mm
  },
  p2state/.style={
    state,
    isosceles triangle,
    isosceles triangle apex angle=60,
    shape border rotate=-90, % triangle pointing DOWN
    minimum size=12mm
  }
}

\newtheorem{Theorem}{Theorem}[section]

\newtheorem{Lemma}[Theorem]{Lemma}

\newtheorem{Definition}[Theorem]{Definition}
\newtheorem{remark}[Theorem]{Remark}
\newtheorem{example}[Theorem]{Example}

\newcommand{\A}{\mathcal{A}}
\renewcommand{\H}{\mathcal{H}}
\newcommand{\B}{\mathcal{B}}
\newcommand{\C}{\mathcal{C}}
\newcommand{\D}{\mathcal{D}}
\newcommand{\I}{\mathcal{I}}
\newcommand{\J}{\mathcal{J}}

\newcommand{\G}{\mathcal{G}}

\newcommand{\K}{\mathcal{K}}

\renewcommand{\P}{\mathcal{P}}

\renewcommand{\S}{\mathcal{S}}
\newcommand{\T}{\mathcal{T}}

\newcommand{\X}{\mathcal{X}}

\newcommand{\NN}{\mathbb{N}}
\newcommand{\PP}{\mathbb{P}}
\newcommand{\EE}{\mathbb{E}}

\newcommand{\RR}{\mathbb{R}}

\newcommand{\proj}{\textnormal{proj}} 
\newcommand{\supp}{\textnormal{supp}}   

\newcommand{\until}{\,..\,}
\newcommand{\eps}{\varepsilon}
\newcommand{\given}{\,|\,}

\DeclareMathOperator*{\argmin}{arg\,min}
\newcommand{\defas}{\coloneqq}

\newcommand{\zeroplus}{\texttt{0}^{+}}
\newcommand{\zeroplusplus}{\texttt{0}^{++}}
\newcommand{\abszero}{\texttt{0}^*}
\newcommand{\oneplus}{\texttt{1}^{+}}
\newcommand{\oneplusplus}{\texttt{1}^{++}}
\newcommand{\onetop}{\texttt{1}^{\top}}
\newcommand{\absone}{\texttt{1}^*}

\newcommand{\Zero}{\textnormal{zero}}
\newcommand{\One}{\textnormal{one}}
\newcommand{\Zeroabs}{\textnormal{zeroabs}}
\newcommand{\Oneabs}{\textnormal{oneabs}}

\allowdisplaybreaks

\title{Approximating the Uniform Value in \\Hidden Stochastic Games with Doeblin Condition}

\author{
    Krishnendu Chatterjee\thanks{Institute of Science and Technology Austria, Austria.}
\and 
    David Lurie\thanks{Paris Dauphine University, PSL Research University, Paris, France, CEREMADE, and NyxAir, Paris, France.}
\and 
    Raimundo Saona\thanks{London School of Economics and Political Science, London, United Kingdom.}
\and 
    Bruno Ziliotto\thanks{Toulouse School of Economics, Université Toulouse Capitole, Institut de Mathématiques de Toulouse, CNRS, France.}
}

\date{\today} 

\begin{document}

\maketitle 
    
\begin{abstract}
    We study \emph{zero-sum two-player hidden stochastic games}, where players receive partial observations of the state.
    We focus on a central solution concept for analyzing long-duration stochastic games: the \emph{uniform value}, a limiting average payoff that both players can guarantee for sufficiently long durations.
    In the general case, prior work provides examples of games that do not have a uniform value.
    Moreover, for the subclass of games that do have a uniform value, there exists no algorithm that approximates it.
    Therefore, we generalize the \emph{Doeblin condition} for Markov chains (which guarantees the existence of a unique invariant measure) to hidden stochastic games.
    Informally, the Doeblin condition for hidden stochastic games requires that, for every way to play the game, there exists a fixed belief such that, no matter the initial belief over the state of the game, after sufficiently many stages, the posterior belief is probably close to this fixed belief.
    Under the Doeblin condition, we prove the existence of the uniform value, provide an algorithm to approximate it, and prove that no algorithm can compute it exactly.
    Then, we identify structural conditions on the transition function that ensure the Doeblin condition holds both in the blind setting, where observations are uninformative, and in the hidden setting, where observations are partially informative.
    When considering games with only one player, namely partially observable Markov decision processes, our results provide a novel subclass in which the uniform value exists and can be approximated, but cannot be computed exactly.
\end{abstract}


\newpage

%%%%%%%
%%%%%%%
%%%%%%%
%%%%%%%
%%%%%%%
%%%%%%%

%\tableofcontents
 
\section{Introduction}\label{Section: Introduction}

%\paragraph{Context}
Zero-sum two-player stochastic games~\cite{shapley1953stochastic} model the strategic interaction of two players in a finite-state environment.
At each stage, both players simultaneously choose public actions, which, together with the current state, determine an immediate reward and the stochastic transition to the next state.
The classic setting is fully observable: players observe the current state at every stage.
These games generalize several well-known models, including one-player stochastic games, or \emph{Markov decision processes} (MDPs)~\cite{puterman1994}, and zero-player stochastic games, or \emph{Markov chains}~\cite{norris1998markov}.
Stochastic games arise in many applications, including economics~\cite{amir2003stochastic}, multi-agent learning~\cite{littman1994markov}, and cyber-security~\cite{aslanyan2016quantitative}.

%\paragraph{Uniform Value}
In the \emph{$n$-stage game}, with $n$ a positive integer, Player $1$ aims to maximize the expected Cesàro mean $\tfrac{1}{n}\sum_{m=1}^{n} G_m$, where $G_m$ denotes the random (because the state transition is stochastic) reward at stage $m$, while Player $2$ aims to minimize it.
Every $n$-stage game admits an $n$-stage value~\cite{mertens2015repeated}, denoted by $v_n$, a quantity that each player can guarantee unilaterally.
Two classical approaches are commonly used to analyze stochastic games with long durations, as $n$ grows.
In the \emph{asymptotic approach}, one studies the sequence of $(v_n)$ as $n$ grows.
Bewley and Kohlberg~\cite{bewley1976asymptotic} proved that $(v_n)$ converges to a limit, denoted by $v_\infty$.
In the \emph{uniform approach}, one seeks strategies that are approximately optimal for all sufficiently large horizons.
Mertens and Neyman~\cite{mertens1981stochastic} established that every stochastic game admits a \emph{uniform value} $v$, that is, for all sufficiently long horizons, Player $1$ (resp., Player $2$) has a strategy that can guarantee at least (resp., at most) $v$ minus (resp., plus) any prescribed margin.
Moreover, the uniform value $v$ coincides with the limit value $v_\infty$.
Practical algorithms to compute or approximate the uniform value of stochastic games are available; see~\cite{oliu2021new} for recent algorithms.

%\paragraph{Partial Observability}
The assumption of full observability is often unrealistic: players usually observe the state only indirectly through signals~\cite{emery2004approximate,young2013pomdp}.
This limitation motivates the study of \emph{hidden stochastic games}~\cite{renault2020hidden}, also known as stochastic games with signals~\cite{solan2016stochastic} or as partially observable stochastic games~\cite{hansen2004dynamic}.
In this setting, players share a common initial belief, i.e., a probability distribution over the state space, and observe each other's actions and signals at every stage.
Hidden stochastic games generalize both stochastic games~\cite{shapley1953stochastic} and partially observable Markov decision processes (POMDPs)~\cite{krishnamurthy2016partially}.

%\paragraph{Impossibilities}
Despite strong existence results for (fully observable) stochastic games, the hidden setting is substantially more challenging.
Strong negative results apply to the class where signals are uninformative (also called blind), in which players gain no information after every stage beyond the actions that were played.
Ziliotto~\cite{ziliotto2016zero} constructed a hidden stochastic game with uninformative signals in which the uniform value fails to exist.
Madani et al.~\cite{madani2003undecidability} showed that approximating the uniform value of one-player hidden stochastic games with uninformative signals (blind MDPs, where the uniform value is known to exist in general~\cite{rosenberg2002blackwell}) is undecidable.
Our objective is to identify a subclass of hidden stochastic games where the uniform value exists and can be approximated algorithmically.

% \paragraph{Inspiration}
We take inspiration from the Doeblin condition for finite Markov chains (see \cite[Chapter 2]{stroockIntroductionMarkovProcesses2014}). In its simplest form, a finite Markov chain satisfies the Doeblin condition if there exists a state that can be reached from every other state with positive probability in a single transition. This condition implies, for example, a geometric convergence to a unique invariant measure.
In a similar spirit, a hidden stochastic game satisfies our Doeblin condition if, by playing the game for long enough, the posterior belief is sure to reach close to a belief that is independent of the initial belief the players started with.

% \paragraph{Previous classes}
Subclasses of hidden stochastic games have been studied before, including \emph{ergodicity} and \emph{primitivity}.
Blind stochastic games are ergodic~\cite{chatterjee2025ergodic} if the influence of actions taken in the distant past vanishes over time.
Hidden stochastic games are \emph{primitive}~\cite{Chatterjee2026mon} if every state is reached from every other state with positive probability, independently of how players act.
Our Doeblin condition extends both ergodicity and primitivity.

\paragraph{Contributions}
Our main contributions are the following:
\begin{itemize}
    \item 
        We introduce the \emph{Doeblin condition} for hidden stochastic games (\Cref{Definition: Doeblin Hidden Stochastic Game}) and prove that this subclass guarantees the existence of the uniform value and its approximability (\Cref{Theorem: Doeblin Hidden Stochastic Games});
    \item 
        We prove that \emph{ergodicity} (\Cref{Definition: Ergodic blind stochastic games}) and \emph{primitivity} (\Cref{Definition: Primitive hidden stochastic game}) imply the Doeblin condition in the blind setting and hidden setting, respectively (\Cref{Theorem: Ergodic Blind is Doeblin} and \Cref{Theorem: Primitive Hidden is Doeblin});
    \item 
        We provide a natural extension of ergodicity to the hidden setting (\Cref{Definition: Ergodic Hidden Stochastic Game}) and show that it fails to guarantee the existence of the uniform value (\Cref{Theorem: weakly ergodic hidden is not doeblin}).
\end{itemize}

\paragraph{Technique}
Every hidden stochastic game is equivalent to a fully observable stochastic game on the infinite belief space~\cite{mertens2015repeated}.
The Doeblin condition can be interpreted as a ``reset'' property for the belief process: after sufficiently many stages, the posterior belief has a strictly positive probability of being close to a fixed belief.
From the so-called \emph{belief stochastic game}, we construct a finite-state stochastic game, which we call \emph{abstract stochastic game}, by discretizing the set of beliefs.
Then, we compare the $n$-stage objectives of the original and abstract games via a coupling argument.

%\paragraph{Coupling}
We construct a coupling between the original hidden stochastic game and the abstract stochastic game.
The coupling is organized into blocks and each block is further divided into smaller sub-blocks on which the Doeblin condition holds.
This organization ensures that, with high probability, the beliefs in the two games become approximately close in the first part of each block.
Once beliefs are close, the coupling forces the same strategy pair to be played in both games, which in turn guarantees that the expected reward on the second part of the block is close.
The lengths of the blocks and sub-blocks are chosen so that the time spent in the first part of each block, when the beliefs may still be far apart, is negligible compared to the time spent in the second part, during which the expected rewards remain close.
Because the difference between expected rewards is small in each block, the coupling ensures that the $n$-stage objectives are close.

%\paragraph{Existence of the Uniform Value}
We deduce the existence of the uniform value in Doeblin hidden stochastic games by its existence in finite-state stochastic games~\cite{mertens1981stochastic}.
Indeed, using the coupling argument, we establish that the $n$-stage payoffs of the game and its abstract version are close.
With this, we show that there is a unique accumulation point of the uniform values of abstract games corresponding to different approximation parameters, and this limit corresponds to the uniform value of the Doeblin hidden stochastic game.

%\paragraph{Approximation of the Uniform Value}
We quantify the distance between the $n$-stage payoffs of the game and its abstract version, uniformly in $n$.
In particular, the uniform value of the abstract game, which can be computed~\cite{oliu2021new}, is an approximation of the uniform value of the Doeblin hidden stochastic game.
We bound the approximation error in terms of the size of the abstract game, which is large enough to apply the Doeblin condition.
We conclude with an explicit algorithm that approximates the uniform value in Doeblin hidden stochastic games, given explicit parameters of the Doeblin condition.

%\paragraph{Blind}
In the blind setting, i.e., when signals are uninformative, we consider the ergodicity condition previously introduced by Chatterjee et al.~\cite{chatterjee2025ergodic}.
This condition formalizes the idea that the belief dynamics forget the distant past, i.e., starting from any two initial beliefs, a sufficiently long sequence of action pairs drives the beliefs to within epsilon of each other.
We show that this property implies the Doeblin condition with explicit parameters, see~\Cref{Theorem: Ergodic Blind is Doeblin}.
Because computing the uniform value is undecidable for ergodic blind stochastic games~\cite[Theorem 3]{chatterjee2025ergodic}, we deduce in \Cref{Theorem: Ergodic Blind is Doeblin} that computing the uniform value in Doeblin hidden stochastic games is undecidable.

%\paragraph{Hidden}
In the hidden setting, we study two structural conditions.
First, we provide a natural extension of the ergodicity condition from the blind case to the hidden case, introducing \emph{weakly ergodic hidden stochastic games} (\Cref{Definition: Ergodic Hidden Stochastic Game}).
Ziliotto~\cite{ziliotto2016zero} previously introduced a hidden stochastic game in which the uniform value does not exist.
Building on this construction, we provide an example of a weakly ergodic hidden stochastic game where the uniform value fails to exist.
Second, we study the primitive condition previously introduced for sets of matrices~\cite{cohen1982sets,seneta2006non}.
For each action pair and signal, there is an update matrix corresponding to the transition probability on the state space after taking an action and receiving a signal.
The primitive condition on this set of matrices formalizes the idea that, after a sufficiently long period, every state-signal pair can occur with positive probability.
We prove that this property implies the Doeblin condition in \Cref{Theorem: Primitive Hidden is Doeblin}, and therefore primitive hidden stochastic games have a uniform value.

\paragraph{Related Literature}
Our results relate to the Doeblin condition for general stochastic systems, the existence of value in games, and the computability of the value of systems that include both stochastic and nondeterministic uncertainty.

%\paragraph{Doeblin Condition}
The Doeblin condition is a well-studied property of Markov chains, even in infinite-state environments~\cite{meyn2012markov,revuz2008markov}.
In the one-player setting, Yushkevich~\cite{yushkevich1997blackwell} proved the existence of the uniform value in MDPs with a Borel state space and compact action spaces under a Doeblin-type condition.
We refer the reader to~\cite{arapostathis1993discrete} for a survey in the one-player setting.
In contrast, results in the two-player setting are less abundant, even in the full information setting. 
Federgruen~\cite{federgruen1978n} studied full information multiplayer stochastic games with countable state space and compact action sets under a Doeblin-type condition.
They proved the existence of a stationary equilibrium for the limsup average reward objective.
However, this condition essentially requires a unichain structure, which is typically a very restrictive assumption for hidden stochastic games.

%\paragraph{Existence}
Counterexamples to the existence of the uniform value in hidden stochastic games were given by Ziliotto~\cite{ziliotto2016zero}.
In fact, nonexistence occurs even in blind stochastic games.
Positive results have been found for special subclasses.
For POMDPs and MDPs, Rosenberg et al.~\cite{rosenberg2002blackwell} and Blackwell~\cite{blackwell1962discrete} respectively proved the existence of the uniform value.
For blind stochastic games, Venel~\cite{venel2015commutative} established the existence of the uniform value under commutativity assumptions, without addressing algorithmic aspects.
Recently, Chatterjee et al.~\cite{chatterjee2025ergodic} proved both the existence of the uniform value and the decidability of its approximation under an ergodicity condition. 
However, they do not consider the general setting with partial observations.

%\paragraph{Approximability}
The decidability of computing or approximating the value of hidden stochastic games has been extensively studied.
However, most results~\cite{belly2025revelations, chatterjee2014partial,chatterjee2013survey,chatterjee2010probabilistic, gimbert2014deciding,asadi2025revealing} concern logical objectives, a different class of objectives; see Chatterjee et al.~\cite{chatterjee2012survey} for a survey.
By contrast, results on the decidability of the uniform value are scarce.
For blind MDPs, the general problem is undecidable~\cite{madani2003undecidability}, and decidability has been shown under strong assumptions~\cite{chatterjee2010probabilistic}.
For blind stochastic games, Chatterjee et al.~\cite{chatterjee2025ergodic} recently established decidability under an ergodicity condition. 
For POMDPs, Chatterjee et al.~\cite{chatterjee2022finite} proved that the single-player case requires only finite-memory strategies for approximating the limit, but this result does not by itself yield decidability.

\paragraph{Novelty}
Our contributions bring several new insights for hidden stochastic games.
First, to the best of our knowledge, this is the first work to use a Doeblin condition for hidden stochastic games to prove the existence of the uniform value and the decidability of its approximation.
Second, our results (\Cref{Theorem: Doeblin Hidden Stochastic Games}) are ``tight'': 
while the exact computation of the uniform value is decidable in stochastic games, and the uniform value in hidden stochastic games does not exist in general, exact computation of the uniform value remains undecidable for Doeblin hidden stochastic games.
This establishes a clear separation, highlighting that Doeblin hidden stochastic games cannot be simply reduced to stochastic games.
Third, because objectives such as reachability~\cite{madani2003undecidability} and other $\omega$-regular objectives~\cite{chatterjee2012survey} can be expressed via the uniform value, our results apply to those objectives as well. 

\paragraph{Outline}
\Cref{Section: Preliminaries} first introduces hidden stochastic games, shows their equivalence with belief stochastic games, and states our computational problems.
\Cref{Section: Class Description and Main Results} defines \emph{Doeblin hidden stochastic games}, \emph{ergodic blind stochastic games}, and \emph{primitive hidden stochastic games}, and presents our main results.
\Cref{Section: Proof of Doeblin Hidden Stochastic Games} establishes the existence of the uniform value and the decidability of its approximation. 
\Cref{Section: proof Primitive and Ergodic are Doeblin} shows that ergodic blind stochastic games and primitive hidden stochastic games satisfy the Doeblin condition. 
Finally, we discuss natural extensions of ergodicity and primitivity conditions in \Cref{Section: Discussion}.

\paragraph{Notation}
Calligraphic letters (e.g., $\I, \J,\H,\K,\S$) denote sets; their elements (e.g., $i$, $j$, $h$, $k$, $s$) appear in lowercase; random variables use uppercase (e.g., $I$, $J$, $H$, $K$, $S$). 
For a finite set $\C$, let $\Delta(\C)$ be the set of probability distributions over $\C$ and let $\delta_{\{c\}}$ denote the Dirac measure at $c \in \C$. 
For integers $a$ and $b$, the notation $[a \until b]$ represents the integer set $\{a,a+1,\ldots,b\}$.
The set of real numbers is denoted by $\mathbb{R}$, while $\NN$ and $\NN^*$ represent the sets of natural numbers and non-zero natural numbers, respectively. 
For a matrix $P \in \RR^{|\K| \times |\K|}$, we write $P > 0$ to indicate that every entry in $P$ is strictly positive and $(P)_{k \in \K}$ to denote its $k$-th column.
Given a vector $b \in \RR^{|\K|}$, $b^\top$ denotes its transpose.

%%%%%%%
%%%%%%%
%%%%%%%
%%%%%%%
%%%%%%%
%%%%%%%

\section{Preliminaries}\label{Section: Preliminaries}

We introduce the class of \emph{two-player zero-sum hidden stochastic games}. 
Then, we describe its reduction to a stochastic game over the set of beliefs.
Finally, we state the exact and approximation problems for hidden stochastic games.

\subsection{Framework}

\paragraph{Game}
A \emph{two-player zero-sum hidden stochastic game}, denoted by $\Gamma$, is defined by a $6$-tuple $\Gamma=(\K,\I,\J,\S,p,g)$, where:
\begin{itemize}
    \item 
        $\K$ is the finite set of states;
    \item 
        $\I$ and $\J$ are the finite sets of actions for Player $1$ and Player $2$, respectively;
    \item 
        $\S$ is the finite set of signals;
    \item 
        $p \colon \K\times \I\times\J\rightarrow\Delta(\K\times \S)$ is the transition probability function;       
    \item 
        $g \colon \K\times\I\times \J\rightarrow[0,1]$ is the stage reward function.
\end{itemize}

\paragraph{Related Models}
Stochastic games~\cite{shapley1953stochastic} are hidden stochastic games in which the observed signal contains the successor state.
Blind stochastic games~\cite{venel2015commutative} are hidden stochastic games in which the set of signals is a singleton. 
In this case, the players are said to be \emph{blind}.

\paragraph{Matrices}
For each action pair $(i,j)\in \I\times \J$ and signal $s\in\S$, define the (sub-stochastic) matrix $P(i,j,s)$ by setting, for all states $k,k'\in\K$,
\[
    P_{k,k'}(i,j,s) \defas p(k',s \given k,i,j) \defas p(k,i,j)(k',s).
\]
Denote by $\P\defas\left\{P(i,j,s)\given (i,j,s)\in\I\times\J\times \S\right\}$ the set of all such sub-stochastic matrices, that is, $\sum_{k'\in \K}P_{k,k'}\le 1$ for all $k\in \K$ and $P_{k,k'}\ge 0$ for all $k,k'\in \K$.
Each matrix $P\in\P$ represents the joint probabilities of transitioning from the current state $k\in \K$ to a subsequent state $k'\in \K$ and observing signal $s\in\S$, given the chosen action pair $(i,j)\in\I\times \J$. 

\paragraph{Dynamic}
A hidden stochastic game starting from $b_1\in \Delta(\K)$, denoted by $\Gamma(b_1)$, evolves as follows. 
An initial state $K_1$ is selected according to $b_1$. 
The players know $b_1$, but do not know $k_1$, the realization of $K_1$.
At each stage $m\ge 1$, Player $1$ and Player $2$ simultaneously select actions $I_m$ and $J_m$, respectively.
A stage reward $G_m\defas g(K_m,I_m,J_m)$ is generated, but \emph{not} observed by the players.
Subsequently, the successor state $K_{m+1}$ and the public signal $S_{m+1}$ are drawn according to $p( \,\cdot \given K_m, I_m, J_m)$.
Finally, the players observe the triple $(I_m,J_m,S_{m+1})$ but not $K_{m+1}$ or $G_m$.

\paragraph{History} 
A \emph{history} before stage $m$ is a sequence $(i_1,j_1,s_2,\ldots,i_{m-1},j_{m-1},s_m)$. 
Let $\H_m\defas (\I\times \J\times\S)^{m-1}$ be the set of histories before stage $m$, with $(\I\times\J\times \S)^{0}\defas\{\emptyset\}$. 

\paragraph{Strategy} 
A (history-dependent) \emph{strategy} for Player $1$ is a mapping $\sigma\colon\bigcup_{m\ge 1}\H_m\to \Delta(\I)$, where $\sigma(i\given h_m) \defas \sigma(h_m)(i)$ is the probability of choosing action $i\in \I$ given the history $h_m \in \H_m$.
Similarly, a strategy for Player $2$ is a mapping $\tau\colon \bigcup_{m\ge 1}\H_m\to \Delta(\J)$, where $\tau(j\given h_m) \defas \tau(h_m)(j)$ is the probability of taking action $j\in \J$ given the history $h_m\in \H_m$.
We denote the set of strategies for Player $1$ and Player $2$ by $\Sigma$ and $\T$, respectively. 

\paragraph{Shift Strategy}
Given $m\in \NN^*$ and $h_m\in\H_m$, we define $h_m$-shift strategies for hidden stochastic games, as previously introduced for the one-player case in~\cite[Definition 6.4, p.13]{chatterjee2022finite}. 
The $h_m$-shift of a strategy $\sigma\in \Sigma$ (resp., $\tau\in \T$) for Player $1$ (resp., Player $2$) is the strategy $\sigma[h_m]$ (resp., $\tau[h_m]$), defined by $\sigma[h_m](h_{m'})\defas\sigma (h_m,h_{m'})$ (resp., $\tau[h_m](h_{m'})\defas\tau(h_m,h_{m'})$) for all $m'\in\NN^*$. 
Intuitively, $\sigma[h_m]$ (resp., $\tau[h_m]$) denotes the continuation of the strategy $\sigma$ (resp., $\tau$) given that the history of the first $m$ stages is $h_m$.

\paragraph{Probability Measure} 
Given an initial belief $b_1\in\Delta(\K)$ and a strategy pair $(\sigma,\tau)\in \Sigma\times \T$, let $\PP_{\sigma,\tau}^{b_1}$ be the induced probability measure over the set of plays $\Omega=(\K\times\I\times\J\times \S)^\NN$. 
Similarly, let $\EE^{b_1}_{\sigma,\tau}$ be the corresponding expectation under this measure. 

\paragraph{Random History}
Given an initial belief $b_1 \in\Delta(\K)$, a strategy pair $(\sigma,\tau) \in \Sigma \times \T$, and a stage $m \in \NN$, define the \emph{random history} before stage $m$ by $H_m \defas(I_1,J_1,S_2,\ldots,I_{m-1},J_{m-1},S_{m})$, which takes values in $\H_m$.

\paragraph{Admissible History}
Given an initial belief $b_1 \in\Delta(\K)$ and a horizon $m \in \NN$, define the set of admissible histories from $b_1 \in \Delta(\K)$ by 
\[
    \H_m(b_1) \defas \left\{ h_m \in \H_m \,\middle|\, \exists (\sigma,\tau) \in \Sigma \times \T \quad \PP^{b_1}_{\sigma, \tau} (H_m=h_m)>0\right\}.
\]

\paragraph{Payoff} 
For a finite horizon $n \in \NN^*$ and a strategy pair $(\sigma,\tau)\in \Sigma\times \T$, the $n$-stage payoff is
\[
    \gamma_{n}(b_1,\sigma,\tau)\defas\EE_{\sigma,\tau}^{b_1}\left(\dfrac{1}{n}\sum_{m=1}^{n} G_m\right).
\]

\paragraph{Uniform Value}
By~\cite{mertens2015repeated}, the $n$-stage game has an $n$-stage value $v_n(b_1)$ given by
\begin{equation*}
    v_n(b_1)\defas\max_{\sigma\in \Sigma}\min_{\tau\in \T}\gamma_{n}(b_1,\sigma,\tau)=\min_{\tau\in \T}\max_{\sigma\in \Sigma}\gamma_{n}(b_1,\sigma,\tau).
\end{equation*}
A hidden stochastic game $\Gamma$ has a \emph{uniform value} $v\colon\Delta(\K)\to [0,1]$ if, for every $b_1\in\Delta(\K)$ and $\eps>0$, there exists a strategy pair $(\sigma^*,\tau^*)\in \Sigma\times \T$ and $\overline{n}\in \NN^*$ such that, for all $n \ge \overline{n}$,
\begin{align*}
    & \forall \tau\in \T\colon\quad\gamma_n (b_1,\sigma^*,\tau) \ge v(b_1) - \eps
\end{align*}
and, 
\begin{align*}
    \forall \sigma\in \Sigma\colon\quad\gamma_n (b_1,\sigma,\tau^*) \le v(b_1) + \eps.
\end{align*}
By Mertens and Neyman~\cite{mertens1981stochastic}, every stochastic game has a uniform value.  
However, by Ziliotto~\cite{ziliotto2016zero}, the uniform value need not exist in hidden stochastic games. 

\subsection{Belief Stochastic Games}

This section highlights the standard reduction of hidden stochastic games to an equivalent stochastic game on the set of beliefs~\cite{mertens2015repeated}.
We first formally define stage beliefs as follows.

\paragraph{Stage Belief}
Given an initial belief $b_1 \in \Delta(\K)$ and an admissible history $h_m \in \H_m(b_1)$, define the belief after observing $h_m$ by
\begin{equation*}
    b^{b_1}_{h_m}(\cdot) \defas \PP_{\sigma,\tau}^{b_1}(K_m = \cdot \given H_m=h_m) ,
\end{equation*}
where $(\sigma, \tau)$ is an arbitrary strategy pair such that $\PP_{\sigma,\tau}^{b_1}(H_m=h_m) > 0$, which exists by admissibility. 
Given an initial belief $b_1 \in \Delta(\K)$ and a strategy pair $(\sigma, \tau)$, the random belief at stage $m$ is defined by
\[
    B_m(\cdot) \defas \PP^{b_1}_{\sigma,\tau}(K_m = \cdot \given H_m) .
\]

\paragraph{Belief Stochastic Game}
Given a hidden stochastic game $\Gamma=(\K,\I,\J,\S,p,g)$ with initial belief $b_1\in\Delta(\K)$, the corresponding \emph{belief stochastic game} is defined by a $5$-tuple $\G=(\Delta(\K),\I,\J,\overline{p},\overline{g})$, where:
\begin{itemize}
    \item 
        $\Delta(\K)$ is the infinite set of beliefs;
    \item
        $\I$ and $\J$ are the respective finite sets of actions of Player $1$ and Player $2$; 
    \item 
        $\overline{p}\colon\Delta(\K)\times \I\times \J\rightarrow\Delta(\Delta(\K))$, written as $\overline{p}(b'\given b,i,j)$, is the belief transition function.
        For every $m\in \NN^*$, the \emph{belief transition function} is defined by
        \begin{align*}
            &\overline{p}(\cdot\given b_m,i_m,j_m)\defas \sum_{s_{m+1} \in \S} \mathbb{1}_{\left\{\psi(b_m,i_m,j_m,s_{m+1})=\cdot\right\}} \, \left( \sum_{k,k'\in\K}P_{k,k'}(i_m,j_m,s_{m+1})b_{m}(k)\right),
        \end{align*}
        where $\psi(b_{m},i_{m},j_{m},s_{m+1})$ is the belief update with 
        \[
            \psi(b_{m},i_{m},j_{m},s_{m+1})(\cdot)\defas \dfrac{\sum\limits_{k\in\K} p(\cdot,s_{m+1}\given k,i_{m},j_{m})b_{m}(k)}{\sum\limits_{k,k'\in\K}p(k',s_{m+1}\given k,i_{m},j_{m})b_{m}(k)};
        \]
    \item 
        $\overline{g}\colon\Delta(\K)\times\I\times \J\rightarrow[0,1]$ is the stage reward defined by, for every $b\in\Delta(\K)$ and $(i,j)\in\I\times\J$,
        \[
            \overline{g}(b,i,j)\defas \sum_{k\in\K} b(k)g(k,i,j).
        \]
\end{itemize}

\paragraph{Payoff and Value}
A belief stochastic game proceeds as a (fully observable) stochastic game in the space of beliefs, and the stage reward is denoted by $\overline{G}_m\defas \overline{g}(B_m, I_m, J_m)$.
Given a finite horizon $n \in \NN^*$, the $n$-stage payoff of the belief stochastic game for a strategy pair $(\sigma,\tau)\in \Sigma\times \T$ is defined by
\begin{equation*}
    \overline{\gamma}_n (b_1,\sigma,\tau)\defas \EE_{\sigma,\tau}^{b_1}\left(\dfrac{1}{n}\sum_{m=1}^{n}\overline{G}_m\right),
\end{equation*} 
and its $n$-stage value exists~\cite{mertens2015repeated} and is defined by 
\[
    \overline{v}_n(b_1)\defas\max_{\sigma\in \Sigma}\min_{\tau\in \T}\overline{\gamma}_n(b_1,\sigma,\tau)=\min_{\tau\in \T}\max_{\sigma\in \Sigma}\overline{\gamma}_n(b_1,\sigma,\tau).
\]
By~\cite{mertens2015repeated}, the $n$-stage value of the belief stochastic game coincides with that of the hidden stochastic game. 
However, because the state space is infinite, one cannot directly apply Mertens and Neyman~\cite{mertens1981stochastic} to conclude that the uniform value exists.

\subsection{Computability}

A \emph{decision problem} determines whether a specific property holds for a given input.
A class of decision problems is said to be \emph{decidable} if there exists an algorithm, i.e., a Turing machine that halts on all inputs and correctly determines whether the property is true or false.
If no such algorithm exists, the class of decision problems is said to be \emph{undecidable}.
We define the exact and approximation problems for the uniform value in hidden stochastic games as follows.

\begin{Definition}[Decision version of computing the uniform value]
    Given a hidden stochastic game $\Gamma$ with initial belief $b_1 \in \Delta(\K)$ and a threshold $x \in [0,1]$, the problem asks whether $v(b_1) > x$ holds, where $v(b_1)$ is the uniform value of $\Gamma$.
\end{Definition}

\begin{Definition}[Decision version of approximating the uniform value]
    Given a hidden stochastic game $\Gamma$ with initial belief $b_1 \in \Delta(\K)$, a threshold $x \in [0,1]$, and an error margin $\eps > 0$, the problem asks whether $v(b_1) > x + \eps$ holds, where $v(b_1)$ is the uniform value of $\Gamma$. Moreover, if $v(b_1) \in [x - \eps, x + \eps]$, then any answer is considered correct.
\end{Definition}

\paragraph{Undecidability}
By~\cite[Theorem 4.4]{madani2003undecidability}, both problems are undecidable in blind MDPs. 
Because a negative result holds for broader classes, the undecidability results carry over to the class of hidden stochastic games where the uniform value exists. 

%%%%%%%
%%%%%%%
%%%%%%%
%%%%%%%
%%%%%%%
%%%%%%%

\section{Main Results}\label{Section: Class Description and Main Results}

We introduce the main class studied in this paper, namely \emph{Doeblin hidden stochastic games}, and state our main results.
Then, we present sufficient conditions on the transitions, namely \emph{ergodicity} and \emph{primitivity}, that guarantee the Doeblin condition.
Define the $L_1$-norm by, for every $b\in\RR^{|\K|}$, $\left\|b\right\|_1\defas \sum_{k\in \K}\left|b(k)\right|$.

\subsection{Doeblin Condition}

This section identifies a general subclass, namely \emph{Doeblin hidden stochastic games}, for which the uniform value exists and whose approximation problem is decidable.

\begin{Definition}[Doeblin hidden stochastic game]
\label{Definition: Doeblin Hidden Stochastic Game}
    A hidden stochastic game $\Gamma$ is Doeblin if, 
    for every $\eps>0$, 
    there exists $m_\eps \in \NN^*$ and $\delta_\eps > 0$ such that,
    for all $\left(\sigma,\tau\right)\in \Sigma\times \T$, 
    there exists $\overline{b}\in \Delta(\K)$ such that,
    for all $b\in \Delta(\K)$, 
    \[
        \PP^{b}_{\sigma,\tau} \left( 
            \left\| B_{m_\eps} - \overline{b} \right\|_1 \le \eps
        \right)
            \ge \delta_\eps.
    \]
\end{Definition}

\paragraph{Main Contribution}
Our main contributions on Doeblin hidden stochastic games are the following.

\begin{Theorem}\label{Theorem: Doeblin Hidden Stochastic Games}
    % For every Doeblin hidden stochastic game:
    % \begin{itemize}
    %     \item 
    %         The uniform value exists and is independent of the initial belief.
    %     \item 
    %         There exists an explicit algorithm to approximate its uniform value, which relies on the mapping $\eps \mapsto (m_\eps, \delta_\eps)$ from the Doeblin condition.
    % \end{itemize}
    The following statements hold:
    \begin{itemize}
        \item 
            Every Doeblin hidden stochastic game admits a uniform value. In particular, the uniform value is independent of the initial belief.
        \item   
            If the mapping $\eps \mapsto (m_\eps, \delta_\eps)$ is computable, then the approximation problem for Doeblin hidden stochastic games is decidable.
    \end{itemize}
\end{Theorem}

% \Cref{Theorem: Doeblin Hidden Stochastic Games} has the following consequence.

% \todo[inline]{@For Rai: When reading this part I'm wondering if simply putting the corollary in the Theorem 3.2 Item 2 would not make it simpler? It may feels weird that an important result is "only" a corollary. I propose:
% - Every Doeblin hidden stochastic game admits a uniform value. In particular, the uniform value is independent of the initial belief.
% - 
% }

% \begin{Corollary}
%     If the mapping $\eps \mapsto (m_\eps, \delta_\eps)$ is computable, then the approximation problem for Doeblin hidden stochastic games is decidable.
% \end{Corollary}

% \todo[inline]{RAI: Evaluate how to map back this result to the decidability of approximation. Because now the theorem does not talk about the decision problem, we should add something that relates them.}

\paragraph{Classic Doeblin Condition}
The Doeblin condition has been extensively studied for Markov chains under the following formulation.
A Markov chain $\Gamma$ satisfies the Doeblin condition if there exists $m \in \NN^*$, $\alpha>0$, and $\nu\in \Delta(\Delta(\K))$ such that, for every $b\in \Delta(\K)$ and measurable set $\A \subseteq \Delta(\K)$,
\[
    \PP^b( B_m \in \A) \ge \alpha \, \nu(\A).
\]
Given $\eps>0$ and $\overline{b}\in \Delta(\K)$, define the closed $\ell_1$-ball of radius $\eps$ centered at $\overline{b}\in \Delta(\K)$ by 
\[
    \C_{\eps}(\overline{b})\defas\left\{b\in \Delta(\K)\,\middle|\, \left\|b-\overline{b}\right\|_1\le \eps\right\}.
\]
%The Doeblin condition is equivalent to the following property.
The Doeblin condition implies the following property.
For every $\eps > 0$, there exist $m_\eps \in \NN^*$, $\delta_\eps > 0$, and $\overline{b}\in \Delta(\K)$ such that, for all $b\in \Delta(\K)$,
\[
    \PP^{b}(B_{m_\eps}\in \C_\eps(\overline{b})) \ge \delta_\eps .
\]
%Indeed, in one direction it is enough to choose some $\eps > 0$ and define $\nu \propto \mathbb{1}_{\C_\eps(\overline{b})} \in \Delta(\Delta(\K))$.
Indeed, %For the reverse direction, 
because $\nu$ is a probability measure on the compact set $\Delta(\K)$, there exists $\overline{b} \in \supp(\nu)$ such that, for every $\eps > 0$, we have that $\nu(\C_\eps(\overline{b})) > 0$.
Then, we conclude by defining $m_\eps \defas m$ and $\delta_\eps \defas \alpha\nu(\C_\eps(\overline{b}))>0$.
% completes the equivalence.

\subsection{Sufficient Conditions}\label{Subsection: sufficient conditions}

We present sufficient conditions for the Doeblin condition by requiring structural properties on the set of matrices $\P$.
We first observe that the belief update can be represented using a forward product of matrices.
Consider a hidden stochastic game $\Gamma=(\K,\I,\J,\S,p,g)$ with an initial belief $b_1\in\Delta({\K})$. 
Define the forward product of matrices by $T\colon\bigcup_{m\in \NN^*}\H_m\to \RR^{|\K|\times |\K|}$ by
\begin{equation*}
    T(h_{m})\defas P(i_1,j_1,s_{2})P(i_2,j_2,s_{3})\cdots P(i_{m-1},j_{m-1},s_{m}).
\end{equation*}
By induction, we have that, for every initial belief $b_1\in\Delta(\K)$, admissible history $h_{m}\in\H_{m}(b_1)$ with $m\ge 1$, and state $k'\in\K$, 
\begin{align*} 
    b_{m}(k')%=\dfrac{\sum_{k\in\K}b_1(k)t_{k,k'}(h_{m})}{\sum_{k,k'\in\K}b_1(k)t_{k,k'}(h_{m})}
    =\dfrac{b_1^\top (T(h_{m}))_{k'}}{b_1^\top T(h_{m})\mathbf{1}}.
\end{align*}

\subsubsection{Ergodicity}

We present the subclass of ergodic blind stochastic games, previously introduced by Chatterjee et al.~\cite{chatterjee2025ergodic}.
This subclass leverages the following definition from~\cite[Definition 4.4, p. 136]{seneta2006non}.

\paragraph{Ergodicity}
A sequence of stochastic matrices $\{ P_m \}_{m\ge 1}$ on $\K\times\K$ is ergodic if, for all $k, \overline{k}, k' \in \K$, 
\begin{equation*}
    \lim_{m\to\infty} (P_1P_2\cdots P_m)_{k,k'} - (P_1P_2\cdots P_m)_{\overline{k},k'} = 0. 
\end{equation*}

\paragraph{Coefficient of Ergodicity}
We introduce the coefficient of ergodicity $\tau_e$ to characterize the ergodic property.
By~\cite{seneta2006non}, given a stochastic matrix $P$, define $\tau_e$ by
\begin{equation*}
    \tau_e(P)\defas\dfrac{1}{2}\max_{k,\overline{k}\in \K}\sum_{k'=1}^{|\K|}\left|P_{k,k'}-P_{\overline{k},k'}\right|.
\end{equation*}
By~\cite[Lemma 4.1, p. 136]{seneta2006non}, ergodicity is equivalent to $\lim_{m\to\infty}\tau_e(P_1P_2\cdots P_m)=0$.\\

We define \emph{ergodic blind stochastic games} as follows.

\begin{Definition}[Ergodic blind stochastic game]
\label{Definition: Ergodic blind stochastic games}
    A blind stochastic game $\Gamma$ is ergodic if, for all $\eps>0$, there exists an integer $m_\eps\in \NN^*$ such that, for every history $h_{m_\eps}\in \H_{m_\eps}$,
    \begin{equation}
        \tau_e(T(h_{m_\eps})) \le \eps .\label{Equation: Conditions for Ergodic BSGs}
    \end{equation}
\end{Definition}

\begin{remark}
    Observe that, in the blind setting, a history is simply a sequence of action pairs.
\end{remark}

\paragraph{Conditions}
Sufficient conditions for the ergodicity condition in blind stochastic games can be found in~\cite{chatterjee2025ergodic}.
For example, a blind stochastic game is said to be Markov if for every $P\in \P$, there exists $k'\in \K$ such that, for every $k\in \K$, we have $P_{k,k'}>0$.
By Chatterjee et al.~\cite{chatterjee2025ergodic}, every Markov blind stochastic game is ergodic.

\paragraph{Contribution}
We show that ergodic blind stochastic games satisfy the Doeblin condition.

\begin{Theorem}\label{Theorem: Ergodic Blind is Doeblin}
    Every ergodic blind stochastic game satisfies the Doeblin condition, with an explicit mapping $\eps \mapsto (m_\eps, \delta_\eps)$.
    In particular, the exact problem in Doeblin hidden stochastic games is undecidable.
\end{Theorem}

\noindent We refer to \Cref{Section: Discussion} for a discussion on extending the ergodicity condition to the hidden setting. 

%%%%%%%%%%%%%%%%%%%%%%%%%%%%%%%%%%%%%%%%%%%%%%%
%%%%%%%%%%%%%%%%%%%%%%%%%%%%%%%%%%%%%%%%%%%%%%%

\subsubsection{Primitivity}

We define the subclass of \emph{primitive hidden stochastic games}, previously introduced in the multiplayer setting by Chatterjee et al.~\cite{Chatterjee2026mon}.
This subclass leverages the following definition, previously introduced by Cohen~\cite{cohen1982sets}.

\paragraph{Primitive}
A set of nonnegative matrices $\{ P_a \}_{a \in \A}$ is primitive if there exists $m \in \NN^*$ such that, for every sequence $(a_1, a_2, \ldots, a_{m}) \in \A^m$, all coordinates of the matrix $P_{a_1} P_{a_2} \cdots P_{a_m}$ are strictly positive. 

\paragraph{Coefficient of Primitivity}
We introduce the coefficient of primitivity $\tau_p$, also called Birkhoff coefficient~\cite{seneta2006non}, to characterize the primitive property.
Consider a nonnegative matrix $P$. 
If the minimum entry of $P$ is zero, then define $\tau_p(P) \defas 1$.
Otherwise, the minimum entry of $P$ is strictly positive, so define 
\begin{align*}
    \tau_p(P)
        \defas \frac
            {1-[\psi(P)]^{1/2}}
            {1+[\psi(P)]^{1/2}},
\end{align*}
where 
\begin{align*}
    \psi(P) \defas \min_{ k, \overline{k} ,k' ,k'' \in \K } 
        \frac
        {P_{ k , k' } P_{ \overline{k} , k'' } }
        {P_{\overline{k},k'} P_{k,k''}}.
\end{align*}
Note that $0 \le \tau_p(P)\le 1$ and that Birkhoff's coefficient is sub-multiplicative, i.e., $\tau_p(P_1P_2) \le \tau_p(P_1)\tau_p(P_2)$ for any two column-allowable matrices $P_1$ and $P_2$; see~\cite[p. 83]{seneta2006non} for a proof.
Observe that primitivity is equivalent to $\lim_{m \to \infty} \max_{(a_1, \ldots, a_m) \in \A^m} \tau_p\left(P_{a_1} P_{a_2} \cdots P_{a_m}\right) = 0$.
Therefore, we define the class of \emph{primitive hidden stochastic games} as follows.

\begin{Definition}[Primitive hidden stochastic game]
\label{Definition: Primitive hidden stochastic game}
    A hidden stochastic game is primitive if,  for all $\eps>0$, there exists an integer $m_\eps\in \NN^*$ such that, for all $h_{m_\eps} \in \H_{m_\eps}$,
    \begin{equation*}
        \tau_p \left(T(h_{m_\eps})\right) \le \eps.
    \end{equation*}
\end{Definition}

\paragraph{Conditions}
Necessary and sufficient conditions for primitivity of a set of nonnegative matrices are provided in~\cite{cohen1982sets}. 
For example, a hidden stochastic game is primitive when every entry of every matrix is strictly positive.

\paragraph{Contribution}
We show that primitive hidden stochastic games satisfy the Doeblin condition.

\begin{Theorem}\label{Theorem: Primitive Hidden is Doeblin}
    Every primitive hidden stochastic game satisfies the Doeblin condition, with an explicit mapping $\eps \mapsto (m_\eps, \delta_\eps)$.
\end{Theorem}

% \todo[inline]{RAI: 
% - Reminder to add the decidability of checking primitivity for hidden stochastic games, only if we decide it to prove it in this paper.
% - Include the computational consequence as a corollary: decidability of approximation for ergodic blind stochastic games, and how it relates to the previous paper.}

%%%%%%%
%%%%%%%
%%%%%%%
%%%%%%%
%%%%%%%
%%%%%%%

\section{Proof of Theorem \ref{Theorem: Doeblin Hidden Stochastic Games}}
\label{Section: Proof of Doeblin Hidden Stochastic Games}

This section proves \Cref{Theorem: Doeblin Hidden Stochastic Games}, i.e., the uniform value exists in Doeblin hidden stochastic games and the approximation problem is decidable.\\

The proof of \Cref{Theorem: Doeblin Hidden Stochastic Games} proceeds as follows:
\begin{itemize}
    \item 
        First, we show that, for a fixed strategy pair, the difference between the expected average starting from two initial beliefs is bounded (\Cref{Lemma: continuity reward}).
    \item 
        Second, we introduce a stochastic game with finite-state and finite-action sets, called \emph{abstract stochastic game}.
    \item
        Third, using a coupling argument, we compare the $n$-stage objectives of the original Doeblin game and the abstract game.
    \item
        Fourth, we prove \Cref{Theorem: Doeblin Hidden Stochastic Games}, adapting the approach used in the proof of Theorem 1 in~\cite{chatterjee2025ergodic}.
\end{itemize}

\paragraph{Different Initial Beliefs} 
We show that the difference in the $n$-stage payoff under the same strategy pair but starting from different initial beliefs can be bounded.

\begin{Lemma}
\label{Lemma: continuity reward}
    Consider a hidden stochastic game $\Gamma$.
    We have that, for every $n \in \NN^*$, strategy pair $(\sigma,\tau)\in \Sigma\times \T$, and arbitrary initial belief $b_1,b_1'\in \Delta(\K)$, 
    \begin{equation}
        \left| \EE_{\sigma,\tau}^{b_1}\left( \dfrac{1}{n}\sum_{m=1}^{n} G_m \right) - \EE_{\sigma,\tau}^{b_1'}\left( \dfrac{1}{n}\sum_{m=1}^{n} G_m \right) \right| \le \left\|b_1-b_1'\right\|_1.
        \label{Equation: Continuity of Reward}
    \end{equation}
\end{Lemma}

\begin{proof}[Proof of Lemma \ref{Lemma: continuity reward}]
    Consider a hidden stochastic game $\Gamma$.
    We have that, for every horizon $n \in \NN^*$, strategy pair $(\sigma,\tau)\in \Sigma\times \T$, and arbitrary pair of beliefs $b_1,b_1'\in \Delta(\K)$,
    \begin{align*}
        &\left| \EE_{\sigma,\tau}^{b_1}\left( \dfrac{1}{n}\sum_{m=1}^{n} G_m \right) - \EE_{\sigma,\tau}^{b_1'}\left( \dfrac{1}{n}\sum_{m=1}^{n} G_m \right) \right| \nonumber\\
        &\qquad=\left|\sum_{k\in \K}b_1(k)\EE_{\sigma,\tau}^{k}\left(\dfrac{1}{n}\sum_{m=1}^{n}G_m\right)-\sum_{k\in \K}b_1'(k)\EE_{\sigma,\tau}^{k}\left(\dfrac{1}{n}\sum_{m=1}^{n}G_m\right)\right|
        &\hspace*{-1cm}\text{(def. of expectation)}\\
        &\qquad=\left|\sum_{k\in \K}(b_1(k)-b_1'(k))\EE_{\sigma,\tau}^{k}\left(\dfrac{1}{n}\sum_{m=1}^{n}G_m\right)\right|
        &\hspace*{-1cm}\text{(rearranging)}\\
        &\qquad\le \sum_{k\in \K}\left|b_1(k)-b_1'(k)\right|\EE_{\sigma,\tau}^{k}\left(\dfrac{1}{n}\sum_{m=1}^{n}G_m\right)
        &\hspace*{-1cm}\text{(triangle inequality)}\\
        &\qquad\le \sum_{k\in \K}\left|b_1(k)-b_1'(k)\right|
        &\hspace*{-1cm}\text{$\left(0\le \EE_{\sigma,\tau}^{k}\left(\dfrac{1}{n}\sum_{m=1}^{n}G_m\right)\le 1\right)$}\\
        &\qquad= \left\|b_1-b_1'\right\|_1,
        &\hspace*{-1cm}\text{(def. of $\|\cdot\|_1$)}
    \end{align*}
    which concludes the proof.
\end{proof}

\paragraph{Discretization of Beliefs}
We introduce a stochastic game with finite-state and finite-action sets, called \emph{abstract stochastic game}.

Consider a Doeblin hidden stochastic game $\Gamma=(\K,\I,\J,\S,p,g)$ with initial belief $b_1 \in \Delta(\K)$ and parameters $\eps > 0$ and $m_\eps \in \NN^*$ given by \Cref{Definition: Doeblin Hidden Stochastic Game}.
For $\eta \in \NN^*$, define the $\eta$-uniform grid of beliefs $\D_\eta \subseteq \Delta(\K)$ by
\begin{align*}
    \D_\eta 
        \defas \left\{ b \in \Delta(\K)\,\middle|\, b(k) = \dfrac{n(k)}{\eta},\, n(k) \in \{0, \ldots, \eta\},\,\sum_{k = 1}^{|\K|} n(k) = \eta \right\}.
\end{align*}
Denote by $\pi_\eta\colon \Delta(\K)\to \D_\eta$ a projection function such that $\pi_\eta(b) \in \argmin_{b'\in \D_\eta}\left\|b-b'\right\|_1$.
By construction, for every $b\in \Delta(\K)$, there exists $\pi_\eta(b) \in \D_\eta$ such that
\[
    \left\|b - \pi_\eta(b) \right\|_1 \le \dfrac{|\K|^2}{\eta}
    \qquad \text{and} \qquad
    \supp(b) = \supp(\pi_\eta(b)) .
\]
In particular, for every $\eps > 0$, if $\eta \ge |\K|^2 \left \lceil \tfrac{1}{\eps} \right \rceil$, then $\left\|b-\pi_\eta(b)\right\|_1\le \eps$ for all $b \in \Delta(\K)$.

\paragraph{Abstract Stochastic Game} 
For a given history $h_m \in \H_m$ and a triple $(i, j, s) \in \I \times \J \times \S$, we denote the concatenation of $(i, j, s)$ to $h_m$ by $h_{m+1} = h_m \times (i, j, s)$.
The \emph{abstract stochastic game} of $\Gamma$ starting from $b_1$ with recall $\eta \in \NN^*$, denoted by $\Gamma_A(b_1,\eta)$, is defined by a $6$-tuple 
\[
    \Gamma_A(b_1,\eta) = \left(\X_\eta,\I,\J,\overline{p}_A,\overline{g}_A,x_1\right),
\]
where:
\begin{itemize}
    \item 
        $\X_\eta$ is the finite set of abstract states, defined by 
        \begin{equation*}
            \X_\eta\defas \{ (b_1, \emptyset) \} \cup \bigcup_{\substack{b \in \D_\eta \\s \in \S}} \{ (b, s) \} \cup \bigcup_{\substack{m \in [2 \until \eta]\\b \in \D_\eta\\}} \left\{ (b, h_m) \,\middle|\,  h_m \in \H_m(b) \right\}.
        \end{equation*}
        Define the function $\proj \colon \X_\eta \to \Delta(\K)$ that associates an abstract state with a belief by
        \begin{align*}
            \proj(x)(\cdot) 
                \defas \begin{cases}
                    b^{b}_{h_m}(\cdot) = \PP^{b}(K_m = \cdot \given H_m = h_m)
                        & x = (b, h_m) \\
                    b(\cdot)
                        & x = (b, s)
                \end{cases} . 
        \end{align*}  
    \item
        $\I$ and $\J$ are the finite sets of actions of Player $1$ and Player $2$, respectively; 
    \item 
        $\overline{p}_A \colon \X_\eta \times \I \times \J \to \Delta(\X_\eta)$ is the abstract transition function.
        For every $x, x' \in \X_\eta$, $(i, j) \in \I \times \J$, define the \emph{abstract transition function} $\overline{p}_A$ by
        \begin{align*}
            \overline{p}_A(x'\given x,i,j)
                \defas \sum_{s \in \S} \mathbb{1}_{\left\{x'=\psi_A(x,i,j,s)\right\}} \cdot \PP(S = s \given \proj(x), i, j) ,
        \end{align*}
        where $\PP(S = s \given \proj(x), i, j) \defas \left( \sum_{k, k' \in \K} \proj(x)(k) \, p(k, i, j)(k', s) \right)$ and $\psi_A \colon \X_\eta \times \I \times \J \times \S\to \X_\eta$ is an \emph{abstract belief update} where $\psi_A(x,i,j,s)$ is defined by
        \begin{align*}
            \left\{
                \begin{array}{ll}
                    \defas \left( b, h_m \times (i, j, s) \right)
                        & x = (b, h_m), m \in [1 \until \eta-1] \\
                    \defas \left( b, (i, j, s) \right)
                        & x = (b, s') \\
                    \in \left\{ (b', s) \,\middle|\, b' \in \argmin\limits_{\substack{\overline{b}\in  \D_\eta\\ \supp(\overline{b}) = \supp\left(b^{b}_{h_{m} \times (i,j,s)}\right)}} \left\|\overline{b} - b^{b}_{h_{m} \times (i,j,s)} \right\|_1 \right\}
                        & x = (b, h_m), m = \eta ;
                \end{array}
            \right.
        \end{align*}
    \item
        $\overline{g}_A \colon \X_\eta \times \I\times \J \to [0,1]$ is the abstract stage reward function. 
        For every state $x\in \X_\eta$ and action pair $(i,j) \in \I\times\J$, the \emph{abstract reward function} is defined by
        \begin{equation*}
            \overline{g}_A(x, i, j)
                \defas \sum_{k \in \K} \proj(x)(k) \, g(k,i,j);
        \end{equation*}
    \item 
        $x_1 = (b_1, \emptyset) \in \X_\eta$ is the initial abstract state.
\end{itemize}
To simplify notation, we may drop the dependence of the abstract game on $b_1$, $\eta$, or both, depending on the context.

\paragraph{Dynamic}
An abstract stochastic game $\Gamma_{A}(b_1, \eta)$ evolves as follows. 
The initial state is $x_1 = (b_1, \emptyset)$. 
The players know $x_1$ and observe the state throughout the game.
At each stage $m \ge 1$, 
\begin{itemize}
    \item 
        Player $1$ and Player $2$ simultaneously select actions $I_m$ and $J_m$, respectively, which are observed by both players;
    \item 
        A stage reward $\overline{G}^A_m\defas \overline{g}_A(X_m, I_m, J_m)$ is generated and the players can compute it;
    \item 
        Subsequently, the successor state $X_{m+1}$ is drawn according to $\overline{p}_A( \,\cdot \given X_m, I_m, J_m)$.
\end{itemize}
In particular, after every block of $\eta$ stages, the state is of the form $x = (b, s)$, for some $b \in \D_\eta$ and $s \in \S$.

\paragraph{History}
Given an abstract stochastic game $\Gamma_{A}(b_1, \eta)$, a \emph{history} before stage $m$ is a sequence $(x_1, i_1,j_1,x_2,\ldots, i_{m-1},j_{m-1},x_m)$.
Denote the set of histories before stage $m$ by $\H_m^{A(\eta)} \defas \{ x_1 \} \times (\I \times \J \times \X_\eta)^{m-1}$, with $\H_1^{A(\eta)} \defas \{ x_1 \}$.

\paragraph{Strategies}
A \emph{(history-dependent) strategy} in $\Gamma_A$ for Player $1$ is a mapping $\sigma_A \colon \bigcup_{m\ge 1} \H_m^A \to \Delta(\I)$.
Similarly, a strategy for Player $2$ in $\Gamma_A$ is a mapping $\tau_A\colon \bigcup_{m\ge 1}\H_m^A\to \Delta(\J)$.
We denote the set of strategies for Player $1$ and Player $2$ in $\Gamma_A$ by $\Sigma_A$ and $\T_A$, respectively. 

\paragraph{Admissible History}
Define the set of \emph{admissible histories} before stage $m$ by
\[
    \H_m^{A(\eta)}(x_1) 
        \defas \left\{ h_m^A \in \H_m^{A(\eta)} \,\middle|\, \exists (\sigma_A, \tau_A) \in \Sigma_A \times \T_A \quad \PP^{x_1}_{\sigma_A, \tau_A} (H_m^A = h_m^A) > 0 \right\}.
\]

\paragraph{Payoff and Value}
The $n$-stage payoff of the abstract stochastic game $\Gamma_{A}$ given by strategy pair $(\sigma_A, \tau_A) \in \Sigma_A \times \T_A$ is defined by
\begin{equation*}
    \gamma_{n}^{A}(x_1,\sigma_A,\tau_A) 
        \defas \EE_{\sigma_A,\tau_A}^{x_1}\left(\dfrac{1}{n}\sum_{m=1}^{n} \overline{G}^A_m\right),
\end{equation*} 
and the $n$-stage value is defined by
\begin{equation*}
    v_{n}^{A}(x_1)
        \defas \max_{\sigma_A\in \Sigma_A} \min_{\tau_A\in \T_A} \gamma_{n}^{A}(x_1,\sigma_A,\tau_A)
        = \min_{\tau_A\in \T_A} \max_{\sigma_A\in \Sigma_A} \gamma_{n}^{A}(x_1,\sigma_A,\tau_A).
\end{equation*}
By~\cite{mertens1981stochastic}, the uniform value exists and is denoted by $v_{A}(x_1)$.

\paragraph{History Mapping}
Given a Doeblin hidden stochastic game $\Gamma$ with initial belief $b_1 \in \Delta(\K)$ and $\eta \in \NN^*$, consider the corresponding abstract stochastic game $\Gamma_A(b_1, \eta)$.
We will construct a coupling between $\Gamma$ and $\Gamma_A$ and therefore need to translate histories from one game to the other.
To do so, we define mappings $\xi_A \colon \bigcup_{m\ge 1} \H_m^A(x_1) \to \bigcup_{m\ge 1} \H_m$ and $\xi \colon \bigcup_{m\ge 1} \H_m(b_1) \to \bigcup_{m\ge 1} \H^A_m$ by recursion as follows.\\

\noindent From $\Gamma_A$ to $\Gamma$:
\begin{itemize}
    \item 
        Base case ($m = 1$):
        Let $h_1^A = x_1$ be the history at stage $m = 1$ in $\Gamma_A$.
        We define $\xi_A(h_1^A) \defas \emptyset$, the only history at stage $m=1$ in $\Gamma$.
    \item 
        Recursion:
        Let $h_{m+1}^A = h_m^A \times (i_m, j_m, x_{m+1}) = (x_1, \ldots, x_m, i_m, j_m, x_{m+1})$ be a history in $\Gamma_A$.
        We define $\xi_A(h_{m+1}^A) \defas \xi_A(h_{m}^A) \times (i_m, j_m, s_{m+1})$, where $x_{m+1} = \psi_A(x_m,i_m,j_m,s_{m+1})$.
\end{itemize}
From $\Gamma$ to $\Gamma_A$:
\begin{itemize}
    \item 
        Base case ($m=1$):
        Let $h_1 = \emptyset$ be the history at stage $m = 1$ in $\Gamma$.
        We define $\xi(h_1) \defas x_1$, the only history at stage $m=1$ in $\Gamma_A$.
    \item 
        Recursion:
        Let $h_{m+1} = h_m \times (i_m, j_m, s_{m+1})$ be a history in $\Gamma$.
        We define $\xi(h_{m+1}) \defas \xi(h_{m}) \times (i_m, j_m, x_{m+1})$, where $x_{m+1}=\psi_A(x_m,i_m,j_m,s_{m+1})$, and $\xi(h_{m}) = (x_1, \ldots, x_m)$.
\end{itemize}
These definitions rely on the abstract belief update $\psi_A$, which requires certain admissibility of the input.
In particular, the definition of $\xi$ needs to be justified.
The required admissibility is provided by the following result.

\begin{Lemma}
\label{Result: Mapping admissibility}
    Consider a Doeblin hidden stochastic game $\Gamma$ with initial belief $b_1 \in \Delta(\K)$ and $\eta \in \NN^*$.
    Then, every admissible history in $\Gamma_A(b_1, \eta)$ and in $\Gamma$ are mapped to each other through $\xi_A$ and $\xi$.
    Formally, we have that,
    \[
        \xi_A \left( \bigcup_{m\ge 1} \H_m^A(x_1) \right) = \bigcup_{m\ge 1} \H_m(b_1)
        \qquad \text{and} \qquad
        \xi \left( \bigcup_{m\ge 1} \H_m(b_1) \right) = \bigcup_{m\ge 1} \H_m^A(x_1) .
    \]
\end{Lemma}

\begin{proof}[Proof of \Cref{Result: Mapping admissibility}]
    Given a Doeblin hidden stochastic game $\Gamma$ with initial belief $b_1 \in \Delta(\K)$ and $\eta \in \NN^*$, consider the corresponding abstract stochastic game $\Gamma_A(b_1,\eta)$.
    The proof follows by induction on $m\in \NN^*$.
    The base case is given by definition.
    
    For the inductive case of $\xi_A$, consider $h_{m+1}^A = h_m^A \times (i_m, j_m, x_{m+1}) = (x_1, i_1,j_1,\ldots, x_m, i_m, j_m,\linebreak x_{m+1})$ an admissible history in $\Gamma_A$.
    In particular, $h_m^A$ is admissible in $\Gamma_A$.
    By inductive hypothesis, $\xi_A(h_{m}^A) \in \H_m(b_1)$.
    In particular, the belief $b^{b_1}_{\xi_A(h_{m}^A)}$ is well-defined.
    We have to prove that $\xi_A(h_{m}^A) \times (i_m, j_m, s_{m+1}) \in \H_{m+1}(b_1)$, where $s_{m+1}$ is given by $x_{m+1} = \psi_A(x_m, i_m, j_m, s_{m+1})$.
    In other words, we have to prove that $\PP\left(S_{m+1} = s_{m+1} \given b^{b_1}_{\xi_A(h_{m}^A)}, i_m, j_m\right) > 0$.
    Because $h_{m+1}^A \in \H_{m+1}^A(x_1)$, we have that $\overline{p}_A(x_{m+1} \given x_m, i_m, j_m) > 0$.
    Therefore, $\PP(S = s_{m+1} \given \proj(x_m), i_m, j_m) > 0$.
    We conclude since $\supp(\proj(x_m)) = \supp\left(b^{b_1}_{\xi_A(h_{m}^A)}\right)$, which is also given by induction, the definition of $\proj$ and $\psi_A$.

    For the inductive case of $\xi$, consider $h_{m+1} = h_m \times (i_m, j_m, s_{m+1})$ an admissible history in $\Gamma$.
    In particular, $h_m$ is admissible in $\Gamma$.
    By inductive hypothesis, $\xi(h_{m}) = (x_1, \ldots, x_m) \in \H_m^A(x_1)$.
    We have to prove that $\xi(h_{m+1}) = \xi(h_{m}) \times (i_m, j_m, x_{m+1}) \in \H^A_{m+1}(x_1)$, where $x_{m+1}=\psi_A(x_m,i_m,j_m,s_{m+1})$.
    In other words, we have to prove that $\PP(S_{m+1} = s_{m+1} \given \proj(x_m), i_m, j_m) > 0$.
    Because $h_{m+1} \in \H_{m+1}(b_1)$, we have that $\PP\left( S = s_{m+1} \given b^{b_1}_{h_m}, i_m, j_m\right) > 0$.
    We conclude since $\supp\left(b^{b_1}_{h_m}\right) = \supp\left(\proj(x_m)\right)$, which is also given by induction, the definition of $\proj$ and $\psi_A$.
\end{proof}

\paragraph{Relationship between Payoffs} 
We translate strategies between $\Sigma_{A(\eta)}$ and $\Sigma$, in both directions, preserving approximately the same $n$-stage payoff as follows.

\begin{Lemma}
\label{Result: n-stage values are close}
    Consider a Doeblin hidden stochastic game $\Gamma$, an initial belief $b_1 \in \Delta(\K)$, and a parameter $\eps > 0$.
    For every $\eta \in \NN^*$ sufficiently large, the following symmetric properties hold
    \begin{itemize}
        \item 
            For every $\sigma_A \in \Sigma_A$, 
            there exists $\sigma \in \Sigma$ such that, 
            for every $\tau \in \T$, 
            there exists $\tau_A \in \T_A$ such that, 
            for every horizon $n \in \NN^*$,
            \begin{equation}
                \left|\gamma_n(b_1,\sigma,\tau)-\gamma_{n}^{A(\eta)}(x_1,\sigma_A,\tau_A)\right| 
                    \le \eps.
                    \label{Equation: Coupling Player 1 abstract to real} 
            \end{equation}
        \item 
            For every $\tau_A \in \T_A$, 
            there exists $\tau \in \T$ such that, 
            for every $\sigma \in \Sigma$, 
            there exists $\sigma_A \in \Sigma_A$ such that, 
            for every horizon $n \in \NN^*$,
            \begin{equation}
                \left|\gamma_n(b_1,\sigma,\tau)-\gamma_{n}^{A(\eta)}(x_1,\sigma_A,\tau_A)\right| 
                    \le \eps.
                    \label{Equation: Coupling Player 2 abstract to real}
            \end{equation}
        \item 
            For every $\sigma \in \Sigma$, 
            there exists $\sigma_A \in \Sigma_A$ such that, 
            for every $\tau_A \in \T_A$, 
            there exists $\tau \in \T$ such that, 
            for every horizon $n \in \NN^*$,
            \begin{equation}
                \left|\gamma_n(b_1,\sigma,\tau)-\gamma_{n}^{A(\eta)}(x_1,\sigma_A,\tau_A)\right| 
                    \le \eps.
                    \label{Equation: Coupling Player 1 real to abstract} 
            \end{equation}
        \item 
            For every $\tau \in \T$, 
            there exists $\tau_A \in \T_A$ such that, 
            for every $\sigma_A \in \Sigma_A$, 
            there exists $\sigma \in \Sigma$ such that, 
            for every horizon $n \in \NN^*$,
            \begin{equation}
                \left|\gamma_n(b_1,\sigma,\tau)-\gamma_{n}^{A(\eta)}(x_1,\sigma_A,\tau_A)\right| 
                    \le \eps.
                    \label{Equation: Coupling Player 2 real to abstract}
            \end{equation}
    \end{itemize}
\end{Lemma}
% \comment[id=R]{We use all four. For the existence of the uniform value, we translate from $\Gamma_A$ to $\Gamma$. For the approximation of the uniform value, we translate strategies the other way around. Let me know if you find a better way to write this.}

\begin{proof}[Proof of \Cref{Result: n-stage values are close}]
    The proofs of \eqref{Equation: Coupling Player 1 abstract to real}, \eqref{Equation: Coupling Player 2 abstract to real}, \eqref{Equation: Coupling Player 1 real to abstract}, \eqref{Equation: Coupling Player 2 real to abstract} are completely symmetric, so we only give the explicit proof of \eqref{Equation: Coupling Player 1 abstract to real}.   
    
    Consider a Doeblin hidden stochastic game $\Gamma$, an initial belief $b_1\in\Delta(\K)$ and $\eps > 0$.
    Denote the parameters given by the Doeblin condition at $\eps$ by $m_\eps$ and $\delta_\eps$.
    Choose $\omega \in \NN^*$ large enough such that 
    \[
        \left(1 - \delta_\eps^2\right)^\omega \le \eps
    \]
    and $\omega \, m_\eps \ge |\K|^2$. 
    Consider $\eta \defas \omega \, m_\eps \left\lceil \tfrac{1}{\eps}\right\rceil$ and the abstract stochastic game $\Gamma_{A(\eta)}$.
    
    We divide time into blocks of length $\eta$.
    For each $\ell \ge 0$, the $\ell$-th block consists of the stages between $\ell \, \eta + 1$ and $(\ell+1) \, \eta$. 
    Each block is further subdivided into sub-blocks of length $m_\eps$.
    For each $r \in [0 \until \omega \lceil 1/\eps \rceil - 1]$, the $(\ell, r)$-th sub-block consists of stages between $\ell \, \eta + r\, m_{\eps} + 1$ and $\ell \, \eta + (r+1) \, m_{\eps}$.
    Denote the first stage of $(\ell, r)$ sub-block by $m_{\ell, r} \defas \ell \, \eta + r \, m_\eps + 1$.

    Intuitively, the coupling proceeds by blocks and sub-blocks as follows.
    First, if the beliefs in the two games are epsilon close at the start of a sub-block, then the same strategy pair, i.e., function of histories, is played in both games for the remainder of the block.
    By \Cref{Lemma: continuity reward}, the average payoff on the rest of the block will be epsilon close in both games.
    Second, if the beliefs differ by more than epsilon at the start of a sub-block, then the same strategy pair is played within the sub-block in both games.
    By the Doeblin condition, at the end of the sub-block the beliefs are close with constant probability.
    Because there are many sub-blocks, there is a high probability that the beliefs become close early in a block, and therefore having close payoffs in each block.

    Fix a strategy $\sigma_A \in \Sigma_A$ in $\Gamma_A$. 
    We use the coupling to construct a strategy $\sigma \in \Sigma$ in $\Gamma$.
    Similarly, fix a strategy $\tau \in \T$ in $\Gamma$.
    We use the coupling to construct a strategy $\tau_A \in \T_A$ in $\Gamma_A$.
    The coupling is defined by a probability measure $\nu$ on the product history space $\left(\bigcup_{m\ge 1} \H_m(b_1) \right)\times \left(\bigcup_{m\ge 1} \H_m^A(x_1) \right)$ such that the first marginal of $\nu$ is the law of the process in $\Gamma(b_1)$ induced by $(\sigma, \tau) \in \Sigma \times \T$, and the second marginal of $\nu$ is the law of the process in $\Gamma_A(x_1, \eta)$ induced by $(\sigma_A, \tau_A) \in \Sigma_A \times \T_A$. 
    We write $\EE_\nu$ and $\PP_\nu$ for the respective expectation and probability measures.

    The coupling is constructed sequentially on $\ell$.
    The case $\ell = 0$ is simple and presents the general idea of the construction.
    Recall the mappings $\xi$ and $\xi_A$ constructed earlier.
    We define the strategies $\sigma \in \Sigma$ and $\tau_A \in \T_A$ up to histories of length $\eta$ as follows.
    Inductively on $m \in [1 \until \eta]$, take $\left(h_{m},h_{m}^A\right) \in \supp \left( \PP_\nu\left(\left(H_{m},H_{m}^A\right) = \cdot \right) \right)$, and define
    \[
        \sigma(h_m) \defas \sigma_A\left(\xi(h_m)\right)
        \qquad 
        \tau_A(h_m^A) \defas \tau\left(\xi_A\left(h_m^A\right)\right) ,
    \]
    and extend $\sigma$ and $\tau_A$ to histories outside the support according to $\nu$ arbitrarily.
    The coupling measure $\nu$ that satisfies this definition is given by the following dynamic. 
    At stage $m$, 
    draw $i_m$ from $\sigma_A(\xi(h_m))$, $j_m$ from $\tau\left(\xi_A(h_m^A)\right)$, and $s_{m+1}$ from $\PP(\cdot\given b_m,i_m,j_m) = \PP(\cdot\given \proj(x_m),i_m,j_m)$.
    Then,
    \begin{itemize}
        \item 
            In $\Gamma$, 
            Player $1$ selects $i$ according to $\sigma_A(\xi(h_m))$ and 
            Player $2$ selects $j$ according to $\tau(h_m)$.
            Then, we draw $s$ according to $\PP(\cdot \given b_{m}, i, j)$, and 
            extend the history by $(i, j, s)$.
        \item 
            In $\Gamma_A$,
            Player $1$ selects $i'$ according to $\sigma_A\left(h^A_m\right)$ and
            Player $2$ selects $j'$ according to $\tau\left(\xi_A\left(h^A_m\right)\right)$.
            Then, we draw $s'$ according to $\PP(\cdot\given x_m, i', j')$, set $x' \defas \psi_A(x_{m}, i', j', s')$, and extend the history by $(i', j', x')$.
    \end{itemize}

    \noindent Because both games evolve according to $\sigma_A$ and $\tau$ during the first $\eta$ stages, by \Cref{Lemma: continuity reward}, independent of how the strategies are extended to future stages,
    \[
        \left|\EE^{b_1}_{(\sigma,\tau)} \left( \dfrac{1}{\eta}\sum_{m=1}^{\eta} \overline{G}_m \right) - \EE^{x_1}_{(\sigma_A,\tau_A)} \left( \dfrac{1}{\eta}\sum_{m=1}^{\eta} \overline{G}_m^A \right) \right|
            = \left|\EE_\nu\left(\dfrac{1}{\eta}\sum_{m=1}^{\eta} \overline{G}_m - \overline{G}_m^A \right)\right|
            = 0 .
    \]

    The general case for $\ell$ is more involved and requires using sub-blocks. 
    We proceed with the construction sequentially on $r$, the $r$-th sub-block within the $\ell$-th block.
    Consider a realized history $\left(h_{m_{\ell,r}},h_{m_{\ell,r}}^A\right) \in \supp \left( \PP_\nu\left(\left(H_{m_{\ell,r}},H_{m_{\ell,r}}^A\right) = \cdot \right) \right)$.
    Note that, because different signals might have been obtained in the different games, in general, $\xi(h_{m_{\ell,r}}) \not = h_{m_{\ell,r}}^A$ and $\xi_A(h_{m_{\ell,r}}^A) \not = h_{m_{\ell,r}}$. 
    Denote the realized belief at stage $m_{\ell,r}$ in $\Gamma$ by $b_{m_{\ell,r}}\coloneq b^{b_1}_{h_{m_{\ell,r}}}$ and the current state in $\Gamma_A$ by $x_{m_{\ell,r}}$.
    We distinguish two cases:
    
    \begin{itemize}
        \item 
            Case 1: 
            if $\left\|b_{m_{\ell,r}}-\proj(x_{m_{\ell,r}})\right\|_1\le 2 \eps$, then, for the rest of the block we proceed as follows.
            For every $m \in [1 \until \eta - r m_\eps]$ and admissible continuation $h_m$ and $h_m^A$,
            \begin{itemize}
                \item 
                    In $\Gamma$, 
                    Player $1$ selects $i$ according to $\sigma_A\left[h^A_{m_{\ell,r}}\right](\xi(h_m))$ and 
                    Player $2$ selects $j$ according to $\tau\left[h_{m_{\ell,r}}\right](h_m)$.
                    Then, we draw $s$ according to $\PP(\cdot \given b_{m}, i, j)$, and 
                    extend the history by $(i, j, s)$.
                    Note that this defines $\sigma\left(h_{m_{\ell,r}} \times h_m\right)$ conditional on the realized history $\left( h_{m_{\ell,r}}, h_{m_{\ell,r}}^A \right)$.
                    Averaging over the realizations of $h_{m_{\ell,r}}^A$ defines $\sigma\left(h_{m_{\ell,r}} \times h_m\right)$ as a distribution over actions.
                \item 
                    In $\Gamma_A$,
                    Player $1$ selects $i'$ according to $\sigma_A\left[h^A_{m_{\ell,r}}\right](h^A_m)$ and
                    Player $2$ selects $j'$ according to $\tau\left[h_{m_{\ell,r}}\right](\xi_A(h^A_m))$.
                    Then, we draw $s'$ according to $\PP(\cdot\given x_m, i', j')$, set $x' \defas \psi_A(x_{m}, i', j', s')$, and extend the history by $(i', j', x')$.
                    Similar to the case in $\Gamma$, this defines $\tau_A\left(h^A_{m_{\ell,r}} \times h^A_m\right)$ conditional on the realized history $\left( h_{m_{\ell,r}}, h_{m_{\ell,r}}^A \right)$.
            \end{itemize}
        \item
            Case 2: 
            if, for every $r' \le r$, we have that $\left\|b_{m_{\ell,r'}}-\proj(x_{m_{\ell,r'}})\right\|_1 > 2 \eps$, then, for the rest of the sub-block we proceed as in case 1, but only for $m_\eps$ stages.
            Formally, for every $m \in [1 \until m_\eps]$ and admissible continuations $h_m$ and $h_m^A$,
            \begin{itemize}
                \item 
                    In $\Gamma$, 
                    Player $1$ selects $i$ according to $\sigma_A\left[h^A_{m_{\ell,r}}\right](\xi(h_m))$ and 
                    Player $2$ selects $j$ according to $\tau\left[h_{m_{\ell,r}}\right](h_m)$.
                    Then, we draw $s$ according to $\PP(\cdot \given b_{m}, i, j)$, and 
                    extend the history by $(i, j, s)$.
                \item 
                    In $\Gamma_A$,
                    Player $1$ selects $i'$ according to $\sigma_A[h^A_{m_{\ell,r}}](h^A_m)$ and
                    Player $2$ selects $j'$ according to $\tau[h_{m_{\ell,r}}](\xi_A(h^A_m))$.
                    Then, we draw $s'$ according to $\PP(\cdot\given x_m, i', j')$, set $x' \defas \psi_A(x_{m}, i', j', s')$, and extend the history by $(i', j', x')$.
            \end{itemize}
    \end{itemize}
    Extending $\sigma$ and $\tau_A$ to histories that are not admissible arbitrarily, we conclude the construction of $\sigma$ and $\tau_A$ for a general block and therefore for all histories.
    Although the dynamic on how to select actions is very similar in both cases, the emphasis is on which continuation strategy is selected.
    Notably, in case 1, the continuation strategy is the same for the remainder of the block.
    In contrast, in case 2, the continuation strategy changes at the end of the sub-block, i.e., at the next sub-block, a different continuation strategy is to be used.
    Importantly, in case 2, both games are using the same strategy pairs during the $(\ell, r)$ sub-block. 

    Define the first time case 1 is satisfied within the $\ell$-th block by
    \[
        T_\ell
            \defas \inf\left\{ r \in \left[0 \until \omega \lceil 1/\eps \rceil - 1\right] \,\middle|\, \,\left\|B_{m_{\ell,r}}-\proj\left(X_{m_{\ell,r}}\right)\right\|_1 \le 2 \eps \right\}.
    \]
    
    If $T_\ell = r < \infty$, then the strategy pairs in each game are the same for the remainder of the block.
    Formally, conditioning on $\left(h_{m_{\ell,r}}, h^A_{m_{\ell,r}}\right) \in \{ T_\ell = r < \infty \}$, for every admissible history $(h_m, h^A_m)$ with $m \in [1 \until \eta - r m_\eps]$,
    \[
        \sigma [h_{m_{\ell, r}}] (h_m) = \sigma_A [h^A_{m_{\ell, r}}] (\xi (h_m) )
        \qquad \text{and} \qquad 
        \tau_A [h^A_{m_{\ell, r}}] (h^A_m) = \tau [h_{m_{\ell, r}}] (\xi_A (h^A_m) )
    \]
    Therefore, similar to the proof of \Cref{Lemma: continuity reward}, 
    \begin{equation}
        \left| \EE_\nu\left( \dfrac{1}{\eta - r m_\eps} \sum_{m = m_{\ell, r}}^{m_{\ell+1, 0}} \overline{G}_m-\overline{G}_m^A \,\middle|\, T_\ell = r \right) \right| 
            \le 2 \eps .
            \label{Equation: Case 1}
    \end{equation}

    If $T_\ell > r$, then the strategy pairs in each game are the same for the remainder of the $(\ell, r)$ sub-block.
    We use the Doeblin condition to state that case 1 will hold after a few sub-blocks, and deduce that the payoffs for the whole block remain close.
    Note that the Doeblin condition can be applied to the abstract game within a sub-block because, within a block, the dynamic is preserved without approximation.
    Formally, conditioning on $\left(h_{m_{\ell,r}}, h^A_{m_{\ell,r}}\right) \in \{ T_\ell > r \}$, for every pair of admissible histories $(h_m, h^A_m)$ with $m \in [1 \until m_\eps]$,
    \[
        \sigma\left[h_{m_{\ell, r}}\right] (h_m) = \sigma_A \left[h^A_{m_{\ell, r}}\right] \left(\xi (h_m) \right)
        \qquad \text{and} \qquad 
        \tau_A \left[h^A_{m_{\ell, r}}\right] \left(h^A_m\right) = \tau \left[h_{m_{\ell, r}}\right] \left(\xi_A \left(h^A_m\right) \right)
    \]
    Therefore, by the Doeblin condition, see \Cref{Definition: Doeblin Hidden Stochastic Game}, there exists $\overline{b} \in \Delta(\K)$, such that both
    \begin{align*}
        \PP_{\sigma\left[h_{m_{\ell,r}}\right], \tau\left[h_{m_{\ell,r}}\right]}^{b_{m_{\ell,r}}} \left( \left\| B_{m_{\eps}} - \overline{b} \right\|_1 \le \eps  \right) 
            \ge \delta_\eps
        \, \text{ and } \, 
        \PP_{\sigma_A\left[h^A_{m_{\ell,r}}\right], \tau_A\left[h^A_{m_{\ell,r}}\right]}^{x_{m_{\ell,r}}} \left( \left\| \proj(X_{m_{\eps}}) - \overline{b} \right\|_1 \le \eps \right) 
            \ge \delta_\eps 
    \end{align*}
    hold at the same time.
    We deduce that,
    \[
        \PP_{\nu} \left( 
            \left\|B_{m_{\ell,r+1}}-\proj\left(X_{m_{\ell,r+1}}\right)\right\|_1 \le 2 \eps
            \,\middle|\,
            T_\ell > r
            \right)
            \ge \delta_\eps^2 . 
    \]
    With this, we deduce that only few sub-blocks are required to fall back to case 1 with high probability.

    By the definition of $\omega$ and $\eta$,
    \begin{equation}
    \label{Equation: Small stopping time}        
        \PP_{\nu}\left( T_\ell > \omega \right)
            \le (1 - \delta_\eps^2)^\omega 
            \le \eps .
    \end{equation}    
    Therefore, for each block $\ell$,
    \begin{align*}
        &\left|\EE_{\sigma,\tau}^{b_1}\left(\dfrac{1}{\eta}\sum_{m=\ell \eta+1}^{(\ell+1)\eta}\overline{G}_m\right)-\EE_{\sigma_A,\tau_A}^{x_1}\left(\dfrac{1}{\eta}\sum_{m=\ell \eta+1}^{(\ell+1)\eta}\overline{G}^A_m\right)\right| \\
            &\qquad = \left|\EE_\nu\left(\dfrac{1}{\eta}\sum_{m=\ell \eta+1}^{(\ell+1)\eta}\left[\overline{G}_m-\overline{G}_m^A\right]\right)\right| 
                &(\text{def. } \nu) \\
            % &\qquad \le \EE_\nu\left(\dfrac{1}{\eta}\sum_{m=\ell \eta+1}^{(\ell+1)\eta}\left|\overline{G}_m-\overline{G}_m^A\right| \right)
            %     &(\text{triangle inequality}) \\
            &\qquad = \left| \sum_{r = 0}^{\omega} \EE_\nu\left(\dfrac{1}{\eta}\sum_{m=\ell \eta+1}^{(\ell+1)\eta}\overline{G}_m-\overline{G}_m^A \,\middle|\, T_\ell = r \right) \, \PP_\nu (T_\ell = r) \right. \\
                &\qquad \qquad 
                \left. + \EE_\nu\left(\dfrac{1}{\eta}\sum_{m=\ell \eta+1}^{(\ell+1)\eta}\overline{G}_m-\overline{G}_m^A \,\middle|\, T_\ell > \omega \right) \PP_\nu (T_\ell > \omega) \right|
                &(\text{conditioning on } T_\ell) \\
            &\qquad \le \sum_{r = 0}^{\omega} \left| \EE_\nu\left(\dfrac{1}{\eta}\sum_{m=\ell \eta+1}^{(\ell+1)\eta}\overline{G}_m-\overline{G}_m^A \,\middle|\, T_\ell = r \right) \right| \, \PP_\nu (T_\ell = r)  \\
                &\qquad \qquad 
                + \left| \EE_\nu\left(\dfrac{1}{\eta}\sum_{m=\ell \eta+1}^{(\ell+1)\eta}\overline{G}_m-\overline{G}_m^A \,\middle|\, T_\ell > \omega \right) \right| \, \PP_\nu (T_\ell > \omega) 
                &(\text{convexity}) \\
            &\qquad \le \sum_{r = 0}^{\omega} \left|  \EE_\nu\left(\dfrac{1}{\eta}\sum_{m=\ell \eta+1}^{(\ell+1)\eta}\overline{G}_m-\overline{G}_m^A \,\middle|\, T_\ell = r \right) \right| \, \PP_\nu (T_\ell = r)  \\
                &\qquad \qquad 
                + \PP_\nu (T_\ell > \omega)
                &(g(\cdot) \in [0, 1]) \\
            &\qquad \le \sum_{r = 0}^{\omega}  \left( 
                \dfrac{r m_\eps}{\eta} + \dfrac{1}{\eta} \left| \EE_\nu\left( \sum_{m = m_{\ell, r}}^{m_{\ell+1, 0}} \overline{G}_m-\overline{G}_m^A \,\middle|\, T_\ell = r \right) \right| 
                \right) \, \PP_\nu (T_\ell = r) \\
                &\qquad \qquad 
                + \PP_\nu (T_\ell > \omega)
                &(g(\cdot) \in [0, 1]) \\
            &\qquad \le \sum_{r = 0}^{\omega}  \left( 
                \dfrac{r m_\eps}{\eta} + \dfrac{\eta - r m_\eps}{\eta} \eps 
                \right) \, \PP_\nu (T_\ell = r) 
                + \PP_\nu (T_\ell > \omega)
                &(\text{\cref{Equation: Case 1}}) \\
            &\qquad \le \sum_{r = 0}^{\omega}  \left( 
                \eps + \eps 
                \right) \, \PP_\nu (T_\ell = r) 
                + \PP_\nu (T_\ell > \omega)
                &(\eta \ge r m_\eps / \eps) \\
            &\qquad = 2 \eps \, \PP_\nu (T_\ell \le \omega) 
                + \PP_\nu (T_\ell > \omega)
                &(\text{summing up}) \\
            &\qquad \le 2 \eps \, (1 - \eps) + \eps
                &(\text{\cref{Equation: Small stopping time}}) \\
            &\qquad \le 3 \eps .
    \end{align*}    

    Summing over blocks, we get that, for every $n_1 \in \NN^*$, 
    \[
        \left|\EE_{\sigma,\tau}^{b_1}\left(\dfrac{1}{n_1 \eta}\sum_{m=1}^{n_1 \eta} \overline{G}_m\right) - \EE_{\sigma_A,\tau_A}^{x_1}\left(\dfrac{1}{n_1 \eta} \sum_{m=1}^{n_1 \eta} \overline{G}_m^A \right) \right| 
            \le 3 \eps.
    \]
    To conclude the statement for all $n$, and not only multiples of $\eta$, consider $\eta' \ge \eta \lceil 1 / \eps \rceil = \omega \, m_\eps \lceil 1 / \eps \rceil^2$ and the abstract stochastic game $\Gamma_{A(\eta')}$.
    Then, let $n \in \NN^*$ be arbitrary.
    Note that, if $n \le \eta'$, then the statement is direct because $\Gamma$ and $\Gamma_{A(\eta')}$ have the same dynamic (encoded differently) up to stage $\eta'$.
    If $n > \eta'$, then rewrite $n$ as blocks of length $\eta$ and write $n = n_1 \, \eta + n_0$, where $n_0 \in [0 \until \eta-1]$.
    After $\eta'$ stages, $\eta$ more stages contribute at most $\eps$ to the payoff, so the payoff for $n_1 \, \eta$ stages and for $n$ stages differ by at most $\eps$.
    Therefore,
    \[
        \left|\EE_{\sigma,\tau}^{b_1}\left(\dfrac{1}{n}\sum_{m=1}^{n} \overline{G}_m\right) - \EE_{\sigma_A,\tau_A}^{x_1}\left(\dfrac{1}{n} \sum_{m=1}^{n} \overline{G}_m^A \right) \right| 
            \le 5 \eps,
    \]
    which concludes the proof.
\end{proof}

We now prove \Cref{Theorem: Doeblin Hidden Stochastic Games}, drawing inspiration from the approach used in the proof of~\cite[Theorem 1]{chatterjee2025ergodic}.

\begin{proof}[Proof of \Cref{Theorem: Doeblin Hidden Stochastic Games}]
    Consider a Doeblin hidden stochastic game $\Gamma$ with initial belief $b_1 \in \Delta(\K)$.
    We prove each statement in turn.

    \paragraph{Existence of the Uniform Value}
    Note that, for all $\eta \in \NN^*$, the abstract game $\Gamma_A(b_1, \eta)$ is a stochastic game with finite-state and finite-action sets.
    By Mertens and Neyman~\cite{mertens1981stochastic}, its uniform value exists, and we denote it by $v_{A(\eta)}$.
    Let $v(b_1)$ be an arbitrary accumulation point of the sequence $\{ v_{A(\eta)} \}_{\eta \in \NN^*}$. 
    We show that $v(b_1)$ is the uniform value of $\Gamma$, i.e., for every $\eps > 0$, each player can uniformly guarantee $v(b_1)$ up to $\eps$ in $\Gamma(b_1)$.    

    Consider $\eps > 0$.
    For every $\eta \in \NN^*$, by definition of the uniform value, there exist a strategy pair $(\hat\sigma_A,\hat\tau_A) \in \Sigma_A \times \T_A$ and $n_{\eta, \eps} \in \NN^*$, such that, for all $n \ge n_{\eta, \eps}$ and $(\sigma_A, \tau_A) \in \Sigma_A \times \T_A$,
    \begin{align}
        \gamma_{n}^{A(\eta)}(x_1, \hat\sigma_A, \tau_A) 
            \ge v_{A(\eta)} - \eps
        \qquad \text{ and } \qquad
        \gamma_{n}^{A(\eta)}(x_1,\sigma_A,\hat\tau_A) 
            \le v_{A(\eta)} + \eps .
            \label{Equation: Uniform strategy}
    \end{align}
    
    \subparagraph{Player~$1$'s guarantee.} 
    By \Cref{Result: n-stage values are close}, \Cref{Equation: Coupling Player 1 abstract to real}, 
    there exists $\eta_\eps \in \NN^*$ such that, 
    for every $\eta \ge \eta_\eps$, 
    for every $\sigma_A \in \Sigma_A$ in $\Gamma_{A(\eta)}$, 
    there exists $\sigma_\eps \in \Sigma$ in $\Gamma$ such that, 
    for every $\tau \in \T$ in $\Gamma$, 
    there exists $\tau_A \in \T_A$ in $\Gamma_{A(\eta)}$ such that, 
    for every $n \in \NN^*$,
    \begin{equation}        
        \left|\gamma_n(b_1, \sigma_\eps, \tau) - \gamma_{n}^{A(\eta)}(x_1, \sigma_A, \tau_A)\right| 
            \le \eps .
            \label{Equation: Player 1 guarantee}
    \end{equation}
    
    Choose $\eta_\eps^* \ge \eta_\eps$ such that $v_{A(\eta_\eps^*)} \ge v(b_1) - \eps$. 
    Consider the strategy $\hat\sigma_{A(\eta_\eps^*)}$, given by \Cref{Equation: Uniform strategy}.
    Consider the corresponding strategy $\sigma_\eps^* \in \Sigma$ in $\Gamma$ given by \Cref{Equation: Player 1 guarantee} for $\sigma_A = \hat\sigma_{A(\eta_\eps^*)}$.
    We show that $\sigma_\eps^*$ guarantees $v(b_1) - 3 \eps$ in $\Gamma$.
    Indeed, for every $\tau \in \T$ in $\Gamma$, and every $n \ge n_{\eta_\eps^*, \eps}$, we have that there exists $\tau_A \in \T_A$ in $\Gamma_{A(\eta_\eps^*)}$ given by \Cref{Equation: Player 1 guarantee} such that
    \begin{align*}
        \gamma_n(b_1, \sigma_\eps^*, \tau) 
            &\ge \gamma_{n}^{A(\eta_\eps^*)}(x_1, \hat\sigma_{A(\eta_\eps^*)}, \tau_A) - \eps 
                &(\text{\cref{Equation: Player 1 guarantee}}) \\
            &\ge v_{A(\eta_\eps^*)} - 2 \eps
                &(\text{\cref{Equation: Uniform strategy}}) \\
            &\ge v(b_1) - 3 \eps .
                &(\text{def. } \eta_\eps^*)
    \end{align*}
    Therefore, Player $1$ uniformly guarantees $v(b_1) - 3 \eps$ in $\Gamma(b_1)$.

    \subparagraph{Player~$2$'s guarantee.} 
    The argument is symmetric to Player~$1$'s.
    Formally, by \Cref{Result: n-stage values are close}, \Cref{Equation: Coupling Player 2 abstract to real}, 
    there exists $\eta_\eps \in \NN^*$ such that, 
    for every $\eta \ge \eta_\eps$, 
    for every $\tau_A\in \T_A$, 
    there exists $\tau \in \T$ such that, 
    for every $\sigma \in \Sigma$, 
    there exists $\sigma_A \in \Sigma_A$ such that, 
    for every horizon $n \in \NN^*$,
    \begin{equation}
        \left|\gamma_n(b_1,\sigma,\tau)-\gamma_{n}^{A(\eta)}(x_1,\sigma_A,\tau_A)\right| 
            \le \eps.
            \label{Equation: Player 2 guarantee} 
    \end{equation}

    Choose $\eta_\eps^* \ge \eta_\eps$ such that $v_{A(\eta_\eps^*)} \le v(b_1) + \eps$. 
    Consider the strategy $\hat\tau_{A(\eta_\eps^*)}$, given by \Cref{Equation: Uniform strategy}.
    Consider the corresponding strategy $\tau_\eps^* \in \T$ in $\Gamma$ given by \Cref{Equation: Player 2 guarantee} for $\tau_A = \hat\tau_{A(\eta_\eps^*)}$.
    We show that $\tau_\eps^*$ guarantees $v(b_1) + 3 \eps$ in $\Gamma$.
    Indeed, for every $\sigma \in \Sigma$, and every $n \ge n_{\eta_\eps^*, \eps}$, we have that there exists $\sigma_A \in \Sigma_A$ given by \Cref{Equation: Player 2 guarantee} such that
    \begin{align*}
        \gamma_n(b_1, \sigma, \tau_\eps^*) 
            &\le \gamma_{n}^{A(\eta_\eps^*)}(x_1, \sigma_A, \hat\tau_{A(\eta_\eps^*)}) + \eps 
                &(\text{\cref{Equation: Player 2 guarantee}}) \\
            &\le v_{A(\eta_\eps^*)} + 2 \eps
                &(\text{\cref{Equation: Uniform strategy}}) \\
            &\le v(b_1) + 3 \eps .
                &(\text{def. } \eta_\eps^*)
    \end{align*}
    Therefore, Player $2$ uniformly guarantees $v(b_1) + 3 \eps$ in $\Gamma(b_1)$.
    In conclusion, $\Gamma(b_1)$ has uniform value $v(b_1)$.

    \paragraph{Independence of the Uniform Value}
    Consider an alternative initial belief $b_1' \in \Delta(\K)$.
    We show that $v(b_1) = v(b_1')$ by connecting their respective abstract stochastic games. 
    For $\eps > 0$, consider two abstract stochastic games $\Gamma_A(b_1, \eta_\eps)$ and $\Gamma_A(b_1', \eta_\eps')$ such that 
    \[
        | v(b_1) - v_{A(b_1, \eta_\eps)} | \le \eps
        \qquad \text{and} \qquad 
        | v(b_1') - v_{A(b_1', \eta_\eps')} | \le \eps .
    \]
    Moreover, consider $\eta_\eps$ and $\eta_\eps'$ large enough such that, through a similar coupling argument as in the proof of \Cref{Result: n-stage values are close}, 
    \[
        | v_{A(b_1, \eta_\eps)} - v_{A(b_1', \eta_\eps')} | \le \eps .
    \]
    % \comment[id=R]{This sounds informal. @David, let me know if you feel comfortable with not giving more details here.}
    % \comment[id=D]{If it's fine for you, I agree to delete it, it's too informal }
    % \comment[id=R]{What do you mean delete it? How would you propose to do the proof?}
    Therefore, $| v(b_1) - v(b_1') | \le 3 \eps$.
    Because $\eps > 0$ is arbitrary, we deduce that $v(b_1)=v(b_1')$, which concludes the proof.

    \paragraph{Approximation of the Uniform Value}
    We provide an explicit algorithm to approximate the uniform value $v$ of $\Gamma$, relying on the mapping $\eps \mapsto (m_\eps, \delta_\eps)$ from the Doeblin condition.
    Consider $\eps > 0$, and suppose that we are given $(m_\eps, \delta_\eps)$ from the Doeblin condition.
    Choose $\omega_\eps \in \NN^*$ large enough such that $(1 - \delta_\eps^2)^{\omega_\eps} \le \eps$ and $\omega_\eps \, m_\eps \ge |\K|^2$, and define $\eta_\eps \defas \omega_\eps \, m_\eps \lceil 1 / \eps \rceil^2$.
    We show that the uniform value of the abstract stochastic game $\Gamma_{A(\eta_\eps)}$ is a good approximation of $v$.
    Formally, we show that
    \[
        | v_{A(\eta_\eps)} - v | \le \eps .
    \]
    Given this inequality, we conclude by approximating $v_{A(\eta_\eps)}$ using, for example, the algorithm described in~\cite[Algorithm 1]{oliu2021new}.

    We show that, for all $n \in \NN^*$,
    \[
        | v^{A(\eta_\eps)}_n(x_1) - v_n(b_1) | \le \eps ,
    \]
    and conclude by taking $n \to \infty$.
    Consider $n \in \NN^*$.
    Take $\sigma_n^*$ an optimal strategy for the $n$-stage payoff in $\Gamma$.
    By \Cref{Result: n-stage values are close}, \Cref{Equation: Coupling Player 1 real to abstract}, 
    there exists $\sigma_A \in \Sigma_A$ such that, 
    for every $\tau_A \in \T_A$, 
    there exists $\tau \in \T$ such that, 
    \[
        \left| \gamma_n(b_1, \sigma_n^*, \tau) - \gamma_{n}^{A(\eta_\eps)}(x_1, \sigma_A, \tau_A) \right| 
            \le \eps .
    \]
    In particular,
    \[
        v_n(b_1) 
            = \min_{\tau \in \T} \gamma_n(b_1, \sigma_n^*, \tau)
            \le \min_{\tau_A \in \T_A} \gamma_{n}^{A(\eta_\eps)}(x_1, \sigma_A, \tau_A) + \eps
            \le v^{A(\eta_\eps)}_n(x_1) + \eps .
    \]
    Similarly, taking $\tau_n^*$ an optimal strategy for the $n$-stage payoff in $\Gamma$, by \Cref{Result: n-stage values are close} and \Cref{Equation: Coupling Player 2 real to abstract}, we have that 
    \[
        v_n(b_1) \ge v^{A(\eta_\eps)}_n(x_1) - \eps ,
    \]
    which concludes the proof.
\end{proof}

Our approximation scheme is detailed in Algorithm \ref{Algorithm: Hidden stochastic games}.

\begin{algorithm}
    \begin{algorithmic}[1]
        \Require Doeblin hidden stochastic game $\Gamma=(\K, \I, \J, \S, p, g)$ and $\eps > 0$.
        \Ensure $v$ is an additive approximation of the uniform value of $\Gamma$ up to $\eps$.
        \State Query $(m_\eps, \delta_\eps)$ from the Doeblin condition.
        \State $\omega_\eps \gets \left \lceil \max \{ \log(\eps) / \log(1 - \delta_\eps^2),  |\K|^2 / m_\eps \} \right \rceil$.
        \State $\eta_\eps \gets \omega_\eps \, m_\eps \lceil 1 / \eps \rceil^2$.
        \State $v \gets v_{A(\eta_\eps)}(x_1)$, the uniform value of the abstract stochastic game $\Gamma_A(b_1, \eta_\eps)$.
        \State \Return $v$
    \end{algorithmic}  
    \caption{Approximation of the uniform value of Doeblin hidden stochastic games}
    \label{Algorithm: Hidden stochastic games}
\end{algorithm}

%%%%%%%
%%%%%%%
%%%%%%%
%%%%%%%
%%%%%%%
%%%%%%%

\section{Proofs of Theorem \ref{Theorem: Ergodic Blind is Doeblin} and Theorem \ref{Theorem: Primitive Hidden is Doeblin}}\label{Section: proof Primitive and Ergodic are Doeblin}

\subsection{Ergodic Blind Stochastic Games}
We prove \Cref{Theorem: Ergodic Blind is Doeblin}, i.e., ergodic blind stochastic games satisfy the Doeblin condition. 
In particular, we show that the exact problem in Doeblin hidden stochastic games is undecidable.

\begin{proof}[Proof of \Cref{Theorem: Ergodic Blind is Doeblin}]    
    Consider an ergodic blind stochastic game $\Gamma$.
    Fix $\eps > 0$. 
    We show that there exist explicit $m_\eps \in \NN^*$ and $\delta_\eps > 0$ such that,
    for all $\left(\sigma, \tau\right) \in \Sigma \times \T$, 
    there exists $\overline{b}\in \Delta(\K)$ such that,
    for all $b \in \Delta(\K)$, 
    \[
        \PP^{b}_{\sigma,\tau} \left( 
            \left\| B_{m_\eps} - \overline{b} \right\|_1 \le \eps
        \right)
            \ge \delta_\eps.
    \]

    By the ergodicity condition, see \Cref{Definition: Ergodic blind stochastic games}, applied to $\eps / 2 > 0$, there exists an integer $m_\eps\in \NN^*$ such that, for every history $h_{m_\eps}\in \H_{m_\eps}$,
    %for all action pair sequences $a^{m_\eps} \in (\I \times \J)^{m_\eps}$, 
    \[
        \tau_e(T(h_{m_\eps})) \le \frac{\eps}{2} .
    \]
    In particular, for all initial beliefs $b, b'\in \Delta(\K)$,
    \begin{align*}
        \left\|b^{b}_{h_{m_\eps}} - b^{b'}_{h_{m_\eps}}\right\|_1
            &= \left\|b^\top T\left(h_{m_\eps}\right) - (b')^\top T\left(h_{m_\eps}\right)\right\|_1\\
            &\le \tau_e\left(T\left(h_{m_\eps}\right)\right)\|b-b'\|_1\\
            &\le \eps .
    \end{align*}
    
    Define $\delta_\eps \defas |\I \times \J|^{- m_\eps} > 0$.
    Consider $\left(\sigma, \tau\right) \in \Sigma \times \T$ arbitrary. 
    Note that, by definition of $\delta_\eps$, there exists $h_{m_\eps} \in H_{m_\eps}$ such that 
    \[
        \PP_{\sigma, \tau}( H_{m_\eps} = h_{m_\eps} ) \ge \delta_\eps .
    \]
    Take an arbitrary belief $b_1 \in \Delta(\K)$ and define $\overline{b}^\top \defas b_1^\top T\left(h_{m_\eps}\right)$.
    Then, for every $b \in \Delta(\K)$
    \begin{align*}
        \PP^{b}_{\sigma,\tau} \left( \left\| B_{m_\eps} - \overline{b} \right\|_1 \le \eps \right)
            &= \PP_{\sigma,\tau} \left( \left\| b^\top T\left(H_{m_\eps}\right) - b_1^\top T\left(h_{m_\eps}\right) \right\|_1 \le \eps \right)\\
            &\ge \PP_{\sigma, \tau}\left( H_{m_\eps} = h_{m_\eps} \right)\\
            &\ge \delta_\eps ,
    \end{align*}
    and the Doeblin condition holds.

    Lastly, the explicit mapping $\eps \mapsto (m_\eps, \delta_\eps)$ is given as follows.
    By~\cite[Corollary 4.6 and Theorem 4.7, p. 90]{paz1971introduction}, taking $m_0 \defas 3^{|\mathcal{K}|}$, we have that, for every $h_{m_0} \in \H_{m_0}$, we get that $\tau_e(T(h_{m_0})) < 1$.    
    Define
    \[
        m_{\eps} \defas \left\lceil \frac{ \ln(\eps/2) }{ \ln (\overline{\tau}(m_0)) } \right\rceil m_0 ,
    \]
    where $\overline{\tau}(m_0) \defas \max \{ \tau_e(T(h_{m_0})) | h_{m_0} \in \H_{m_0} \} < 1$.
    Note that, by submultiplicativity of the $\tau_e$, for every $h_{m_\eps} \in H_{m_\eps}$, we have that
    \[
        \tau_e\left( T\left(h_{m_\eps}\right)\right)
            \le \left(\overline{\tau}(m_0)\right)^{\left\lceil \ln(\eps/2) / \ln (\overline{\tau}(m_0)) \right\rceil}
            \le \eps/2 ,
    \]
    so $m_\eps$ satisfies the definition of ergodicity for $\eps/2$.
    Together with $\delta_\eps\coloneqq\left(|\I \times \J|\right)^{-m_\eps} > 0$, this provides an explicit mapping $\eps \mapsto (m_\eps, \delta_\eps)$.

    Finally, we deduce that the exact problem is undecidable in Doeblin hidden stochastic games as follows.
    By~\cite{chatterjee2025ergodic}, computing the uniform value in Markov blind MDPs is undecidable.
    By~\cite{chatterjee2025ergodic}, every Markov blind MDP is ergodic and thus satisfies the Doeblin condition.
    Therefore, the exact problem in Doeblin hidden stochastic games is undecidable.
\end{proof}

\subsection{Primitive Hidden Stochastic Games}
We prove \Cref{Theorem: Primitive Hidden is Doeblin}, i.e., primitive hidden stochastic games satisfy the Doeblin condition.

% \todo[inline]{David: I proposed a new proof below.
% Let me know what you think.
% Previous attempts are commented below.}
% \todo[inline]{RAI: I give a counter-proposal. I think it is shorter and clearer. In any case, your proposal (which is below) led me to new questions about the definition of primitivity. Please have a look.}

\begin{proof}[Proof of \Cref{Theorem: Primitive Hidden is Doeblin}]
    
    Consider a primitive hidden stochastic game $\Gamma$ and fix $\eps > 0$.
    We show that there exists explicit $m_\eps \in \NN^*$ and $\delta_\eps > 0$ such that,
    for all $\left(\sigma, \tau\right) \in \Sigma \times \T$, 
    there exists $\overline{b}\in \Delta(\K)$ with the following property:
    for all $b \in \Delta(\K)$, 
    \[
        \PP^{b}_{\sigma,\tau} \left( 
            \left\| B_{m_\eps} - \overline{b} \right\|_1 \le \eps
        \right)
            \ge \delta_\eps.
    \]
    By the primitive condition, see \Cref{Definition: Primitive hidden stochastic game}, %applied to $\eps > 0$, 
    there exists $m_\eps\in \NN^*$ such that, for all $h_{m_\eps} \in \H_{m_\eps}$,
    \[
        \tau_p \left(T(h_{m_\eps})\right) \le \eps.
    \]
    In particular, by~\cite{Chatterjee2026mon}, for every pair of initial beliefs $b, b' \in \Delta(\K)$ and $h_{m_\eps}\in \H_{m_\eps}$,
    \[
        \left\| b^{b}_{h_{m_\eps}} - b^{b'}_{h_{m_\eps}} \right\|_1
            \le \tau_p(T(h_{m_\eps}))
            \le \eps .
    \]

    Recall that, for every matrix $T$, if $\tau_p(T) < 1$, then all coordinates of $T$ are strictly positive.
    Therefore,
    \[
        \mu_\eps 
            \defas \min \left\{ \sum_{k' \in \K} T_{k, k'}(h_{m_\eps}) \,\middle|\, \, k \in \K , \, h_{m_\eps} \in \H_{m_\eps} \right\} 
            > 0 \,.
    \]
    Define 
    \[
        \delta_\eps\coloneqq\left(\dfrac{1}{|\I|\times |\J|}\right)^{m_\eps-1} \mu_\eps > 0 \,.
    \]
    Consider an arbitrary strategy pair $\left(\sigma, \tau\right) \in \Sigma \times \T$. 
    We show that there exists a history $h_{m_\eps}^* \in \H_{m_\eps}$ such that, for every belief $b \in \Delta(\K)$,
    \[
        \PP^{b}_{\sigma, \tau}(H_{m_\eps} = h_{m_\eps}^*) 
            \ge \delta_\eps .
    \]

    We construct the history $h_{m_\eps}^*$ inductively as follows. 
    Set $h_1^* = \emptyset$.
    Choose an arbitrary signal $s^* \in \S$.
    For every stage $m \in [1 \until m_\eps-1]$, choose $i_{m}^* \in \I$ and $j_{m}^* \in \J$ such that 
    \[
        \sigma(h_m^*)(i_m^*) \ge \frac{1}{|\I|}
        \qquad \text{and} \qquad
        \tau(h_m^*)(j_m^*) \ge \frac{1}{|\J|} .
    \]
    Then, set $h_{m+1}^* \defas h_{m}^*\times (i_m^*,j_m^*,s^*)$.
    Note that, for every belief $b \in \Delta(\K)$,
    \begin{align*}
        \PP^{b}_{\sigma,\tau} \left( H_{m_\eps} = h_{m_\eps}^* \right)
            &= \sum_{k_1,\ldots,k_{m_\eps}\in \K} b(k_1) \, 
                \prod_{m=1}^{m_\eps-1} \sigma(h_m^*)(i_m^*) \, \tau(h_m^*)(j_m^*) \, 
                P_{k_m,k_{m+1}}(i_m^*, j_m^*, s^*) \\
            &= \left( \prod_{m=1}^{m_\eps-1} \sigma(h_m^*)(i_m^*) \, \tau(h_m^*)(j_m^*) \right) \,
                \sum_{k_1,\ldots,k_{m_\eps}\in \K} b(k_1) \prod_{m=1}^{m_\eps-1} P_{k_m,k_{m+1}}(i_m^*, j_m^*, s^*) \\
            &= \left( \prod_{m=1}^{m_\eps-1} \sigma(h_m^*)(i_m^*) \, \tau(h_m^*)(j_m^*) \right) \,
                b^\top T\left(h_{m_\eps}^*\right)\mathbf{1} \\
            &\ge \left(\dfrac{1}{|\I|\times |\J|}\right)^{m_\eps-1} \mu_\eps\\
            & = \delta_\eps .
    \end{align*} 
    Take an arbitrary belief $b_1 \in \Delta(\K)$ and define $\overline{b}^\top \defas b^{b_1}_{h_{m_\eps}^*}$.
    Then, for every $b \in \Delta(\K)$
    \begin{align*}
        \PP^{b}_{\sigma,\tau} \left( \left\| B_{m_\eps} - \overline{b} \right\|_1 \le \eps \right)
            &\ge \PP^{b}_{\sigma,\tau} \left( H_{m_\eps} = h_{m_\eps}^*, \, \tau_p \left(T(h_{m_\eps}^*)\right) \le \eps \right)\\
            &= \PP^{b}_{\sigma,\tau} \left( H_{m_\eps} = h_{m_\eps}^* \right)\\
            &\ge \delta_\eps ,
    \end{align*}
    so the Doeblin condition holds.

    Lastly, the explicit mapping $\eps \mapsto (m_\eps, \delta_\eps)$ is given as follows.
    By~\cite[Theorem 1, p. 188]{cohen1982sets}, taking $m_0 \defas 2^{|\K|}$, we have that, for every history $h_{m_0} \in \H_{m_0}$, every coordinate of $T(h_{m})$ is strictly positive.
    Therefore, by definition of $\tau_p$, for every $h_{m_0} \in \H_{m_0}$, we get that $\tau_p(T(h_{m_0})) < 1$.
    Define
    \[
        m_{\eps} \defas \left\lceil \ln(\eps) / \ln (\overline{\tau}(m_0)) \right\rceil \, m_0 ,
    \]
    where $\overline{\tau}(m_0) \defas \max \{ \tau_p(T(h_{m_0})) | h_{m_0} \in \H_{m_0} \} < 1$.
    Note that, by submultiplicativity of the $\tau_p$~\cite[p. 83]{seneta2006non}, we have that, for every history $h_{m_\eps} \in \H_{m_\eps}$,
    \[
        \tau_p\left(T\left(h_{m_\eps}\right)\right)
            \le \left(\overline{\tau}(m_0)\right)^{\left\lceil \ln(\eps) / \ln (\overline{\tau}(m_0)) \right\rceil}
            \le \eps,
    \]
    so $m_\eps$ satisfies the definition of primitivity.
    Together with 
    \[
        \delta_\eps\coloneqq\left(\dfrac{1}{|\I|\times |\J|}\right)^{m_\eps-1} \mu_\eps > 0 \,,
    \]
    where $\mu_\eps \defas \min \left\{ \sum_{k' \in \K} T_{k, k'}(h_{m_\eps}) \,\middle|\, \, k \in \K , \, h_{m_\eps} \in \H_{m_\eps} \right\} > 0 $, this provides an explicit mapping $\eps \mapsto (m_\eps, \delta_\eps)$.
\end{proof}

%%%%%%%
%%%%%%%
%%%%%%%
%%%%%%%
%%%%%%%
%%%%%%%

\section{Discussion}\label{Section: Discussion}

In this section, we discuss natural extensions of the sufficient conditions presented in \Cref{Subsection: sufficient conditions} and highlight potential directions for future research.

\subsection{Relaxing Primitivity and Ergodicity Conditions}

We first introduce the subclass of \emph{ergodic hidden stochastic games} in the next definition.

\begin{Definition}[Ergodic hidden stochastic game]\label{Def: Ergodic HSG}
    A hidden stochastic game $\Gamma$ is ergodic if, for every $\eps>0$, there exists $m_\eps\in \NN^*$ such that, for every $b_1, b_1'\in \Delta(\K)$ and $h_{m_\eps}\in \H_{m_\eps}$, we have that $h_{m_\eps}\in \H_{m_\eps}(b_1) \cap \H_{m_\eps}(b_1')$ and 
    \begin{equation}
        \left\|b^{b_1}_{h_{m_\eps}}-b^{b_1'}_{h_{m_\eps}}\right\|_1 
            \le \eps.\label{Equation: forgetting property}
    \end{equation}
\end{Definition}

In the blind setting, \Cref{Def: Ergodic HSG} is equivalent to \Cref{Definition: Ergodic blind stochastic games}.
Recall that, in a blind stochastic game, every history is admissible from every initial belief.
Therefore, for every $b_1,b_1'\in \Delta(\K)$ and $h_{m_\eps}\in \H_{m_\eps}$, we have that $h_{m_\eps}\in \H_{m_\eps}(b_1) \cap \H_{m_\eps}(b_1')$.
Moreover, the ergodicity condition \eqref{Equation: Conditions for Ergodic BSGs} is equivalent to \eqref{Equation: forgetting property}.
Assume first that \eqref{Equation: Conditions for Ergodic BSGs} holds. 
Then, for every $\eps>0$, there exists $m_\eps\in \NN^*$ such that, for every $b_1,b_1'\in \Delta(\K)$ and $h_{m_\eps}\in \H_{m_\eps}$, we have that 
\[
    \left\|b_1^\top T(h_{m_\eps})-(b_1')^\top T(h_{m_\eps})\right\|_1\le \tau_e(T(h_{m_\eps}))\left\|b_1-b_1'\right\|_1\le 2\eps.
\]
Conversely, assume that \eqref{Equation: forgetting property} holds. Then, for every $\eps>0$, there exists $m_\eps\in \NN^*$ such that, we have that, for all $k,k'\in \K$ and $h_{m_\eps}\in \H_{m_\eps}$,
\[
    \left\|\delta_{\{k\}}^\top T(h_{m_\eps})-\delta_{\{k'\}}^\top T(h_{m_\eps})\right\|_1\le \eps.
\]
Taking the maximum over $\K$ yields $\max_{k,k'\in \K}\left\|\delta_{\{k\}}^\top T(h_{m_\eps})-\delta_{\{k'\}}^\top T(h_{m_\eps})\right\|_1=2\tau_e(T(h_{m_{\eps}}))$. 
Therefore, $\tau_e(T({h_{m_\eps}}))\le \eps/2$ which proves the equivalence. 

In the hidden setting, ergodic hidden stochastic games strictly generalize primitive hidden stochastic games.
Indeed, by the proof of \Cref{Theorem: Primitive Hidden is Doeblin}, primitivity implies the admissibility of histories, i.e., for every $b_1,b_1'\in \Delta(\K)$ and $h_{m_\eps}\in \H_{m_\eps}$, we have that $h_{m_\eps}\in \H_{m_\eps}(b_1) \cap \H_{m_\eps}(b_1')$.
Moreover, by Chatterjee et al.~\cite{Chatterjee2026mon}, we have that, for every $\eps>0$, there exists $m_\eps\in \NN^*$ such that, for every $b_1,b_1'\in \Delta(\K)$ and history $h_{m_\eps}\in \H_{m_\eps}$, we have that,
\[
    \left\|b^{b_1}_{h_{m_\eps}}-b^{b_1'}_{h_{m_\eps}}\right\|_1\leq \eps.
\]
Because the proof of \Cref{Theorem: Primitive Hidden is Doeblin} relies on these exact two properties, it follows that ergodic hidden stochastic games satisfy the Doeblin condition.
Finally, because Markov blind MDPs are ergodic but not primitive in general, it follows that the converse implication fails.

\begin{example}
    We construct a Doeblin POMDP $\Gamma$ that is neither ergodic nor primitive.
    Consider $\Gamma=(\K,\I,\S,p,g)$, where: 
    \[
        \K=\{k_1,k_2\},\qquad \I=\{i\},\qquad \S=\{s_1,s_2\}. 
    \]
    The reward function $g$ is arbitrary.
    The transition probabilities are given by
    \[
        p(k_1,s_1\given k_1,i)=\dfrac{1}{2},\qquad p(k_1,s_2\given k_1,i)=\dfrac{1}{2},
    \]
    and
    \[
        p(k_1,s_1\given k_2,i)=\dfrac{1}{2},\qquad p(k_2,s_2\given k_2,i)=\dfrac{1}{2}.
    \]
    Equivalently, the transition matrices are
    \[
        P(i,s_1)=\begin{pmatrix}
            1/2 & 0 \\
            1/2 & 0
        \end{pmatrix}\quad\text{ and }\quad
        P(i,s_2)=\begin{pmatrix}
            1/2 & 0 \\
            0 & 1/2
        \end{pmatrix}.
    \]

    The POMDP $\Gamma$ is not primitive.
    Indeed, for every $m\in \NN^*$, the product of matrices $P(i,s_2)^m$ is not positive.
    Moreover, the POMDP $\Gamma$ is not ergodic. 
    Indeed, fix $b_1=\delta_{\{k_1\}}$, $b_1'=\delta_{\{k_2\}}$, and $\eps=1/2$.
    For every $m\in \NN^*$, consider the admissible history consisting only of signal $s_2$, namely $h_m=(i,s_2,\ldots,i,s_2)$.
    Then, we have that
    \[
        \left\|b_{h_m}^{b_1}-b_{h_m}^{b_1'}\right\|_1=\|\delta_{k_1}-\delta_{k_2}\|_1>\eps.
    \]
    
    Finally, $\Gamma$ satisfies the Doeblin condition.
    Indeed, for every belief $b\in \Delta(\K)$, the probability of observing signal $s_1$ after playing action $i$ is equal to $1/2$.
    Therefore, for every belief $b\in \Delta(\K)$ and $\eps>0$,
    \[
        \PP^b_i\left(\left\|B_2-\delta_{\{k_1\}}\right\|_1\leq \eps\right)\geq 1/2
    \]
    We deduce that $\Gamma$ is Doeblin with parameters $m_\eps=2$, $\delta_\eps=1/2$, and $\overline{b}=\delta_{k_1}$.
\end{example}

%\paragraph{Weakly ergodic hidden stochastic games}
We can relax the ergodicity condition by dropping the admissibility property, introducing weakly ergodic hidden stochastic games as follows.

\begin{Definition}[Weakly ergodic hidden stochastic game]\label{Definition: Ergodic Hidden Stochastic Game}
    A hidden stochastic game $\Gamma$ is weakly ergodic if, for all $\eps>0$, there exists $m_\eps\in \NN^*$ such that, for every $b_1, b_1'\in \Delta(\K)$ and every $h_{m_\eps}\in \H_{m_\eps}(b_1) \cap \H_{m_\eps}(b_1')$ ,
    \[
        \left\|b^{b_1}_{h_{m_\eps}}-b^{b_1'}_{h_{m_\eps}}\right\|_1 
            \le \eps .
    \]
\end{Definition}

The next result shows that not all weakly ergodic hidden stochastic games have a uniform value.
In particular, weakly ergodic hidden stochastic games do not satisfy the Doeblin condition in general.

\begin{Theorem}\label{Theorem: weakly ergodic hidden is not doeblin}
    There exists a weakly ergodic hidden stochastic game that does not have a uniform value.
\end{Theorem}

\begin{proof}[Proof of \Cref{Theorem: weakly ergodic hidden is not doeblin}]
    Our argument builds on the counterexample given by Ziliotto~\cite{ziliotto2016zero}. 
    Consider a hidden stochastic game $\Gamma$, where:
    \begin{itemize}
        \item 
            $\K=\left\{\zeroplus,\zeroplusplus, \abszero, \oneplus, \oneplusplus, \onetop, \absone\right\}$.
        \item 
            $\I=\J=\{c,q\}$.
        \item 
            $\S=\S_1\times\S_2$ with $\S_1=\{d,d'\}$ and $\S_2=\{\Zero,\One,\Zeroabs,\Oneabs\}$.   
            The set of signals $\S_2$ partitions the state space $\K$ into blocks $\B_0=\{\zeroplus,\zeroplusplus\}$, $\B_1=\{\oneplus,\oneplusplus,\onetop\}$, $\B_{\abszero}=\{\abszero\}$, and $\B_{\absone}=\{\absone\}$. Define the block signal function $\zeta\colon \K\to \S_2$ by 
            \[
                \zeta(k)=
                \begin{cases}
                    \One & \text{if } k\in \B_1,\\
                    \Zero & \text{if } k\in \B_0,\\
                    \Oneabs & \text{if } k\in\B_{\absone},\\
                    \Zeroabs & \text{if } k\in \B_{\abszero}.
                \end{cases}
            \]
            Formally, at each stage $m\in\NN^*$, the public signal is a pair $(s_{1,m},s_{2,m})\in \S_1\times \S_2$ where $s_{2,m}=\zeta(k_{m})$, which reveals the block of $k_{m}$.
        \item 
            The stage reward function is $1$ in states $\absone$, $\oneplusplus$, $\oneplus$, and $\onetop$, and $0$ in states $\abszero$, $\zeroplusplus$, and $\zeroplus$, independently of the actions.
    \end{itemize}
    The transitions are defined in \Cref{Figure: Counterexample of HSG}.
    An arrow from a state $k\in \K$ to $k'\in \K$ labeled $(i,p,s_1)\in \{c,q\}\times [0,1]\times \{d,d'\}$ means that, if the player who controls state $k$ plays action $i$, then with probability $p$ the state is $k'$ and the signal $s_1$ is observed. In particular, the states $\abszero$ and $\absone$ are absorbing states. Moreover, Player $1$ controls the states in $\B_0$, while Player $2$ controls the states in $\B_1$.\\
    
    The dynamics of the game present the following features:
    \begin{itemize}
        \item 
            Playing action $q$ either switches between the non-absorbing blocks $\B_0$ and $\B_1$, or goes to $\B_{\abszero}$ and $\B_{\absone}$.
        \item 
            The only non-absorbing next beliefs are the following:
            \begin{align*}
                0_m&\defas2^{-m}\delta_{\{\zeroplusplus\}}+\left(1-2^{-m}\right)\delta_{\{\zeroplus\}}\\
                1_{2m}&\defas2^{-2m}\delta_{\{\oneplusplus\}}+\left(1-2^{-2m}\right)\delta_{\{\oneplus\}}\\
                1_{2m+1}&\defas2^{-2m}\delta_{\{\onetop\}}+\left(1-2^{-2m}\right)\delta_{\{\oneplus\}}.
            \end{align*}
            Moreover:
            \begin{itemize}
                \item 
                    If the belief is $0_m$ and Player $1$ plays $c$, then with probability $1/2$ the signal is $d$ and the next belief is $0_{m+1}$, while with probability $1/2$, the signal is $d'$ and the next belief ``resets'' to $0_0=\delta_{\{\zeroplusplus\}}$.
                \item
                    If the belief is $1_m$ and Player $2$ plays $c$, then with probability $1/2$ the signal is $d$ and the next belief is $1_{m+1}$, while with probability $1/2$ the signal is $d'$ and the next belief ``resets'' to $1_0=\delta_{\{\oneplusplus\}}$.
                \item
                    From $0_m$, playing action $q$ yields $d$ with probability $1-2^{-m}$ and moves to $1_0=\delta_{\{\oneplusplus\}}$, or $d'$ with probability $2^{-m}$ and goes to $\delta_{\{\abszero\}}$. 
                \item 
                    From $1_{2m}$ or $1_{2m+1}$, playing action $q$ yields $d$ with probability $1-2^{-2m}$ and moves to $0_0=\delta_{\{\zeroplusplus\}}$, or $d'$ with probability $2^{-2m}$ and absorbs into $\delta_{\{\absone\}}$.
            \end{itemize}
    \end{itemize}
    \begin{center}
    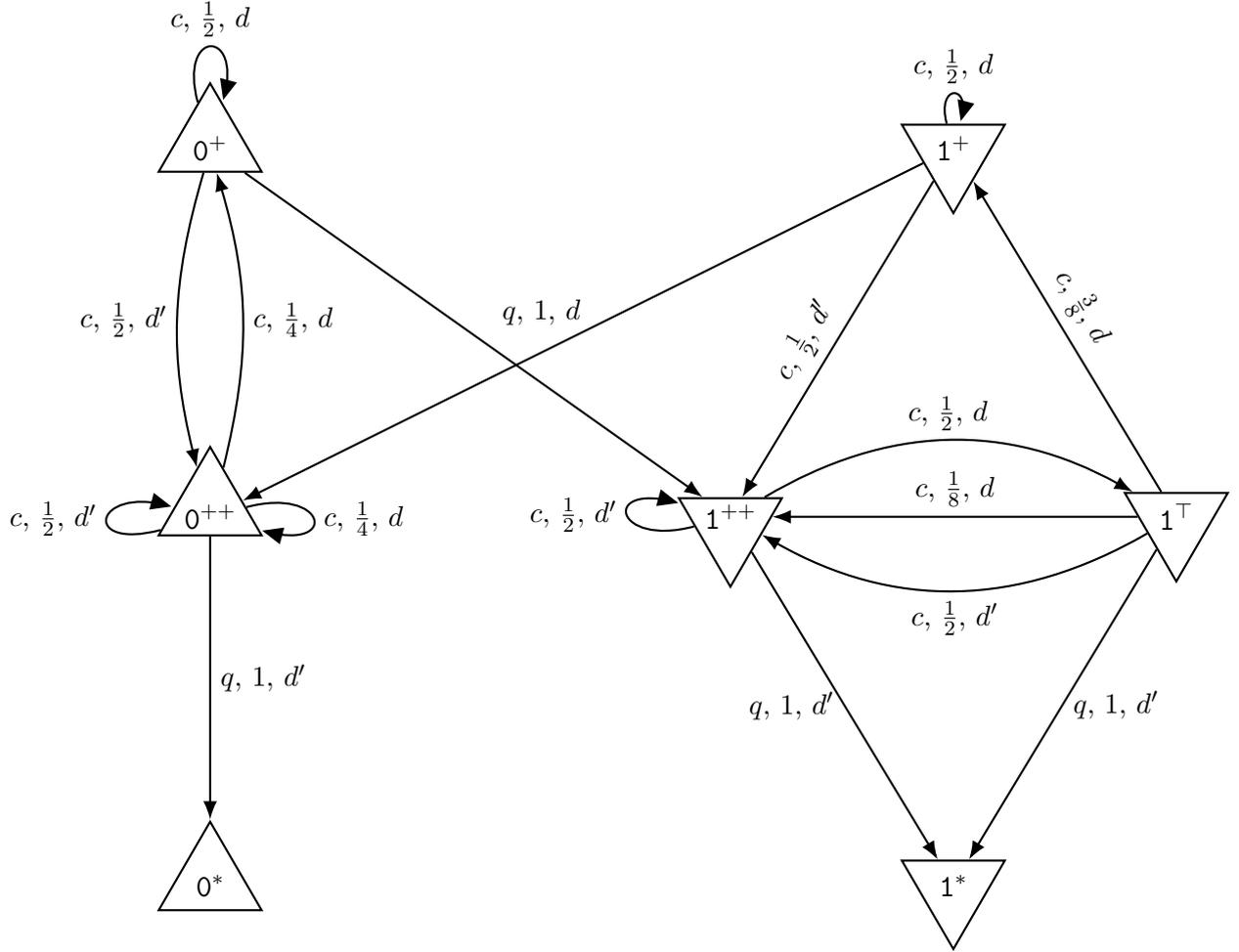
\begin{figure}[ht]
        \begin{tikzpicture}
            
            \begin{scope}[local bounding box=diagram]
                
                \node[p1state] (zplus) at (-6, 5) {$\zeroplus$};
                \node[p2state] (oplus) at (4, 5) {$\oneplus$};
                
                \node[p1state] (zpp)   at (-6, 0) {$\zeroplusplus$};
                \node[p2state] (opp)   at (1, 0) {$\oneplusplus$};
                \node[p2state] (t)     at (7, 0) {$\onetop$};
                
                \node[p1state] (zstar) at (-6,-5) {$\abszero$};
                \node[p2state] (ostar) at (4,-5) {$\absone$};
                
                \path
                (zplus) edge[loop above] node {$c,\,\tfrac12,\,d$} (zplus)
                (zpp) edge[bend right=15] node[right] {$c,\,\tfrac14,\,d$} (zplus)
                (zplus)   edge[bend right=15] node[left]  {$c,\,\tfrac12,\,d'$} (zpp)
                (zpp)   edge[loop right]   node {$c,\,\tfrac14,\,d$} (zpp)
                (zpp)   edge[loop left]    node {$c,\,\tfrac12,\,d'$} (zpp)
                (zpp)   edge               node[right] {$q,\,1,\,d'$} (zstar);
                
                \path
                (oplus) edge[loop above] node {$c,\,\tfrac12,\,d$} (oplus)
                (oplus)   edge node[sloped,above,pos=.55] {$c,\,\tfrac12,\,d'$} (opp)
                (t)     edge node[sloped,above,pos=.55] {$c,\,\tfrac{3}{8},\,d$} (oplus)
                (opp)   edge[bend left=30] node[above] {$c,\,\tfrac12,\,d$} (t)
                (t)     edge node[above] {$c,\,\tfrac18,\,d$} (opp)
                (t)     edge[bend left=30] node[below] {$c,\,\tfrac12,\,d'$} (opp)
                (opp)   edge               node[left]  {$q,\,1,\,d'$} (ostar)
                (opp)   edge[loop left]    node {$c,\,\tfrac12,\,d'$} (opp)
                (t)     edge               node[right] {$q,\,1,\,d'$} (ostar);
                
                \path
                (zplus) edge node[above] {\hspace*{1.5cm}$\;\;\;q,\,1,\,d$} (opp)
                (oplus) edge (zpp);
                            
                % \node[overlay, anchor=south, font=\large\bfseries]
                % at ([yshift=4mm] zplus |- diagram.north) {\quad States controlled by Player $1$};
                
                % \node[overlay, anchor=south, font=\large\bfseries]
                % at ([yshift=4mm] oplus |- diagram.north) {States controlled by Player $2$};
            \end{scope}
            
        \end{tikzpicture}
        \caption{State transition diagram of $\Gamma$. The triangle facing up are states controlled by Player $1$ while the triangle facing down are states controlled by Player $2$.}
        \label{Figure: Counterexample of HSG}
    \end{figure}
    \end{center}

    We prove that $\Gamma$ is weakly ergodic, i.e., for all $\eps>0$, there exists $m_\eps\in \NN^*$ such that, for every $b_1, b_1'\in \Delta(\K)$ and $h_{m_\eps}\in \H_{m_\eps}(b_1) \cap \H_{m_\eps}(b_1')$ ,
    \[
        \left\|b^{b_1}_{h_{m_\eps}}-b^{b_1'}_{h_{m_\eps}}\right\|_1 
            \le \eps .
    \]
    
    Consider $\eps > 0$.
    Take $m_\eps \defas 2 + \lceil \log(1/\eps) \rceil$. 
    We distinguish the following cases:

    \subparagraph{Case A}
    If the players observe $\Zeroabs$ or $\Oneabs$, then both beliefs are the same Dirac from that time and for the subsequent stages.

    \subparagraph{Case B}
    If a player chooses action $q$, then the next belief transitions deterministically as follows: 
    \begin{itemize}
        \item 
            In $\B_0$: observing $d$ leads to $1_0$, while $d'$ to $\delta_{\{\abszero\}}$;
        \item
            In $\B_1$: observing $d$ leads to $0_0$, while $d'$ to $\delta_{\{\absone\}}$.
    \end{itemize}
    Therefore, both processes evolve identically thereafter.

    \subparagraph{Case C}
    By construction, if a player chooses action $c$ and observes $d'$, then the next belief is $0_0$ in $\B_0$, and $1_0$ in $\B_1$.
    Thereafter, all subsequent beliefs are identical.

    \subparagraph{Case D}
    Assume that both players choose action $c$, observe $d$ during $m_\eps$ steps, and the block remains constant, i.e., $s_{2,m}=\Zero$ or $s_{2,m}=\One$ for every $m\in [1\until m_\eps]$. 
    Then,
    \begin{itemize}
        \item 
            In $\B_0$:
            We denote a belief on $\left\{\zeroplusplus,\zeroplus\right\}$ as $(p,1-p)$, where $p=\PP(\zeroplusplus)$. 
            Because under choosing $c$ and observing $d$, the successor belief is $\dfrac{p}{2}\delta_{\{\zeroplusplus\}}+\left(1-\dfrac{p}{2}\right)\delta_{\{\zeroplus\}}$, we get that
            \[
                \left\|b^{b_1}_{h_m}-b^{b_1'}_{h_m}\right\|_1\le \dfrac{2\left|p-p'\right |}{2^m}\le \dfrac{2}{2^m}.
            \]
            Therefore, 
            \[
                \left\|b^{b_1}_{h_{m_\eps}}-b^{b_1'}_{h_{m_\eps}}\right\|_1\le \eps.
            \]
        \item 
            In $\B_1$:
            Recall that the only posterior beliefs with support in $\B_1$ that can appear are $1_{2m}=\tfrac{1}{2^{2m}}\delta_{\{\oneplusplus\}}+\left(1-\tfrac{1}{2^{2m}}\right)\delta_{\left\{\oneplus\right\}}$ and $1_{2m+1}=\tfrac{1}{2^{2m}}\delta_{\left\{\onetop\right\}}+\left(1-\tfrac{1}{2^{2m}}\right)\delta_{\left\{\oneplus\right\}}$.
            Under Player $2$ selecting $c$ and observing $d$, the belief moves from $1_m$ to $1_{m+1}$.
            Let $\alpha(m)$ be the total mass of $1_m$ outside of $\oneplus$. Then, we have that $\alpha(2m)=\alpha(2m+1)=2^{-2m}$ for all $m\in \NN$.
            Therefore, after $m$ steps, the mass outside $\oneplus$ is at most $2^{-2\lfloor m/2 \rfloor}$.
            Take two beliefs $b_1=1_{\ell}$ and $b_1'=1_{\ell'}$ for some $\ell,\ell'\in \NN$. 
            After $m$ steps of $(c,d)$, we reach $b_m=1_{\ell+m}$ and $b_m'=1_{\ell'+m}$.
            Writing $U\defas \left\lfloor\tfrac{\ell+m_\eps}{2}\right\rfloor$ and $V\defas\left\lfloor\tfrac{\ell'+m_\eps}{2}\right\rfloor$, we obtain 
            \begin{align*}
                &\|b_{m_\eps}-b_{m_\eps}'\|_1\\
                &\qquad\le\max\left\{\left\|1_{2U}-1_{2V}\right\|_1,\left\|1_{2U+1}-1_{2V+1}\right\|_1,\left\|1_{2U}-1_{2V+1}\right\|_1\right\}\\
                &\qquad\le\max\left\{2\left|2^{-2U}-2^{-2V}\right|,2\left|2^{-2U}-2^{-2V}\right|,2\max\left\{2^{-2U},2^{-2V}\right\}\right\}\\
                &\qquad\le 2\times 2^{-2\min\{U,V\}}\\
                &\qquad\le \dfrac{4}{2^{m_\eps}}
                    &\hspace*{-2.5cm}\text{($U\ge \lfloor m_\eps/2\rfloor$ and $V\ge \lfloor m_\eps/2\rfloor$)}\\
                &\qquad \le \eps.
            \end{align*}
    \end{itemize}
    Therefore, we deduce that $\Gamma$ satisfies the weak ergodicity condition.

    Note that our formulation differs from Ziliotto's example~\cite{ziliotto2016zero} in that the public signal takes values in $\S_1\times\S_2$, rather than only $\S_1$. 
    The second component $\S_2$ serves to label the block of the next state: by construction, the signal $s_{2,m}=\zeta(k_{m})$ reveals whether $k_{m}\in \{\zeroplus,\zeroplusplus\}, \{\oneplus,\oneplusplus,\onetop\}, \{\abszero\}$ or $\{\absone\}$.
    However, in Ziliotto's example, the players already know at every stage the current block from the history of actions and signals from $\S_1$.
    Indeed, given the current block and the past actions and observed signals, the transition structure deterministically pins down the next block: after playing $c$, the game stays in the same block; after playing $q$ and observing $d$, it switches between $\Zero$ and $\One$, and after playing $q$ and observing $d'$, it moves to the corresponding absorbing block $\Zeroabs$ from $\Zero$ and $\Oneabs$ from $\One$.
    Therefore, the belief dynamics in $\Gamma$ and in Ziliotto's example remain identical.
    We conclude that the uniform value does not exist in $\Gamma$ by~\cite[Theorem 2.5]{ziliotto2016zero}.
\end{proof}

\subsection{Research Directions}

\paragraph{Decidability of verifying the Doeblin condition}
In this paper, we introduced the general subclass of Doeblin hidden stochastic games.
Verifying whether a given hidden stochastic game satisfies the Doeblin condition remains an open problem.
By Chatterjee et al.~\cite{Chatterjee2026mon}, verifying either ergodicity in the blind setting or primitivity in the hidden setting can be done in EXPSPACE.

\paragraph{Exact problem for primitive hidden stochastic games}
To establish undecidability of the exact problem in Doeblin hidden stochastic games, we draw on the fact that the exact problem is undecidable for Markov blind MDPs.
However, the class of Markov blind MDPs does not form a strict subclass of primitive hidden stochastic games. 
Therefore, the exact problem for primitive hidden stochastic games also remains an open problem.

\paragraph{Hidden stochastic games with general sets}
Our abstract stochastic game construction relies on discretizing the belief space.
This naturally suggests studying extensions of hidden stochastic games to more general state spaces, e.g., Euclidean spaces.

%%%%%%%
%%%%%%%
%%%%%%%
%%%%%%%
%%%%%%%
%%%%%%%

\section*{Acknowledgements}     

This material is based upon work supported by the ANRT under the French CIFRE Ph.D. program, in collaboration between NyxAir (France) and Paris-Dauphine University (Contract: CIFRE N° 2022/0513), by the French Agence Nationale de la Recherche (ANR) under reference ANR-21-CE40-0020 (CONVERGENCE project) and ANR-17-EURE-0010 (Investissements d’Avenir program), and partially supported by the ERC CoG 863818 (ForM-SMArt) grant and the Austrian Science Fund (FWF) 10.55776/COE12 grant. 
Part of this work was done at NyxAir (France) by David Lurie. 
Part of this work was done during a one-year visit of Bruno Ziliotto to the Center for Mathematical Modeling (CMM) at University of Chile in 2023, under the IRL program of CNRS.

%%%%%%%
%%%%%%%
%%%%%%%
%%%%%%%
%%%%%%%
%%%%%%%

\bibliographystyle{alpha}
\bibliography{Biblio_new}

\end{document}